\documentclass[11pt]{article}

\usepackage[a4paper]{geometry}
\usepackage{adjustbox}
\usepackage{amssymb}
\usepackage{amsmath,amsfonts,amsthm,amsopn}
\usepackage{latexsym}
\usepackage{mathrsfs}
\usepackage{xcolor}
\usepackage{algorithm}
\usepackage{algorithmic}
\usepackage{subcaption}
\usepackage{dsfont}
\usepackage{multirow}
\usepackage[overload]{empheq}
\usepackage{enumitem}
\usepackage{setspace}
\usepackage{url}
\usepackage{booktabs}
\usepackage{bm}
\usepackage{graphicx}
\usepackage{epstopdf}
\usepackage{ifpdf}
\usepackage[numbers]{natbib}
\usepackage{makecell}
\usepackage{placeins}
\usepackage{soul}
\usepackage{tikz}
\usetikzlibrary{decorations.pathreplacing,calligraphy,calc}
\usetikzlibrary{positioning}
\usetikzlibrary{shapes.geometric,arrows.meta}
\usepackage[colorlinks=true,
            linkcolor=blue,
            citecolor=blue,
            urlcolor=blue]{hyperref}
\usepackage{cleveref}

\AtBeginDocument{\graphicspath{{./}{../}}}
\ifpdf
  \DeclareGraphicsExtensions{.pdf,.png,.jpg,.eps}
\else
  \DeclareGraphicsExtensions{.eps}
\fi

\hypersetup{
  pdftitle={On the Benign Optimization Landscape and Geometry of Heteroscedastic Probabilistic PCA},
  pdfauthor={Peng Wang, Po Chen, Rujun Jiang, and Laura Balzano}
}

\newcommand{\beqa}{\begin{eqnarray}}
\newcommand{\eeqa}{\end{eqnarray}}
\newcommand{\beqas}{\begin{eqnarray*}}
\newcommand{\eeqas}{\end{eqnarray*}}
\newcommand{\ba}{\begin{array}}
\newcommand{\ea}{\end{array}}
\newcommand{\bi}{\begin{itemize}}
\newcommand{\ei}{\end{itemize}}

\newcommand{\RN}[1]{%
  \textup{\uppercase\expandafter{\romannumeral#1}}%
}

\DeclareMathOperator{\diag}{diag}

\newcommand{\dist}{\mathrm{dist}}

\theoremstyle{plain}
\newtheorem{theorem}{Theorem}
\newtheorem{lemma}{Lemma}
\newtheorem{corollary}{Corollary}
\newtheorem{proposition}{Proposition}
\newtheorem{fact}{Fact}

\newtheorem{definition}{Definition}
\newtheorem{assumption}{Assumption}

\theoremstyle{remark}
\newtheorem{remark}{Remark}

\crefname{theorem}{Theorem}{Theorems}
\crefname{lemma}{Lemma}{Lemmas}
\crefname{corollary}{Corollary}{Corollaries}
\crefname{proposition}{Proposition}{Propositions}
\crefname{definition}{Definition}{Definitions}
\crefname{assumption}{Assumption}{Assumptions}
\crefname{fact}{Fact}{Facts}
\crefname{remark}{Remark}{Remarks}
\crefname{figure}{Figure}{Figures}
\crefname{table}{Table}{Tables}

\newcommand{\email}[1]{\protect\href{mailto:#1}{#1}}

\def\bLambda{\bm{\Lambda}}

\def\E{\mathbb{E}}

\def\R{\mathbb{R}}
\def\S{\mathbb{S}}

\newcommand{\blk}{\mathrm{BlkD}}
\begin{document}

\title{On the Benign Optimization Landscape and Geometry of Heteroscedastic Probabilistic PCA\thanks{The first two authors contributed equally to this work.}
}
\author{ 
Peng Wang\thanks{Department of Artificial Intelligence \& Institute of Artificial Intelligence and Brain Sciences, University of Macau, Taipa, Macao SAR, China (\email{pengw@um.edu.mo}).}
\and Po Chen\thanks{School of Data Science, Fudan University, Shanghai, China (\email{chenp24@m.fudan.edu.cn}).}
\and Rujun Jiang\thanks{School of Data Science, Fudan University, Shanghai, China (\email{rjjiang@fudan.edu.cn}).}
\and Laura Balzano\thanks{Department of Electrical Engineering and Computer Science, University of Michigan, Ann Arbor, USA (\email{girasole@umich.edu}).}
}
\date{\today}
\maketitle
\begin{abstract}
In this work, we study the global optimization landscape and local geometry of heteroscedastic probabilistic principal component analysis (HePPCA). Since the signal strengths and noise variances are often unknown in practice, we analyze a formulation that uses fixed plug-in estimates of these parameters. We first show that the population HePPCA problem is an instance of a homogeneous quadratic program with orthogonality constraints (QPOC). For QPOC, we first derive a necessary and sufficient condition for a critical point to be a global maximizer and explicitly construct a positive-curvature direction at every nonglobal critical point. This establishes a benign optimization landscape of QPOC in which every critical point is either a global maximizer or a strict saddle point. Building on this analysis, we establish quantitative curvature bounds for population HePPCA, yielding local geodesic strong concavity near every global maximizer and uniformly positive curvature in suitable tangent directions at all nonglobal critical points. Combining these results with a deterministic perturbation argument and sample covariance concentration bounds, we derive a sufficient sample-size condition under which empirical HePPCA inherits the benign landscape and local geodesic strong concavity with high probability. Finally, we present numerical experiments to corroborate our theoretical results. 
\end{abstract}



\vspace{-0.1in}
\section{Introduction}\label{sec:intro}

Heteroscedastic probabilistic principal component analysis (HePPCA) seeks to recover low-dimensional structure from observations with varying noise levels, as often encountered in applications involving multiple sensors or data sources \cite{gilman2023streaming,hong2021heppcat,xu2023hempcat}. HePPCA and its extensions have found applications in practice, including air quality monitoring \cite{hong2021heppcat}, astronomical spectral analysis \cite{gilman2023streaming,hong2023optimally}, and video motion segmentation \cite{xu2023hempcat}, to name a few.  With the signal strengths and noise variances fixed, maximum likelihood estimation in HePPCA leads to the maximization of a sum of heterogeneous quadratic forms over the Stiefel manifold \cite{gilman2025semidefinite,hong2021heppcat}. Specifically, consider $L$ data groups, where the $l$-th group contains $n_l$ samples for each $l \in [L]$. Let $n := \sum_{l=1}^L n_l$ and let $\bm y_{l,i} \in \R^d$ denote the $i$-th sample in the $l$-th group. The samples $\{\bm y_{l,i}\} \subseteq \R^d$ are generated according to the model
\begin{align}\label{eq:HPCA}
\bm y_{l,i} = \bm Q^\star \mathrm{diag}(\sqrt{\bm \lambda}) \bm z_{l,i} + \bm \delta_{l,i},\quad \forall l \in [L],\ i \in [n_l],
\end{align}
where $\bm Q^\star \in \mathrm{St}(d,K)$ has orthonormal columns spanning the true signal subspace, $\bm \lambda = (\lambda_1,\dots,\lambda_K)$ with $\lambda_1 > \cdots > \lambda_K > 0$ denotes the signal strengths, $\bm z_{l,i} \overset{i.i.d.}{\sim} \mathcal{N}(\bm 0, \bm I_K)$ are latent variables, and $\bm \delta_{l,i} \sim \mathcal{N}(\bm 0, v_l\bm I_d)$ are additive heteroscedastic noise terms with $v_l > 0$. We assume that the noise terms are independent across all samples and independent of all latent variables. According to \cite{gilman2025semidefinite,hong2021heppcat}, a maximum likelihood estimator of the ground truth $\bm Q^\star$ can be obtained by solving the following nonconvex problem:
\begin{align}\label{eq:HePPCA}
    \max_{\bm Q \in \R^{d\times K}} F(\bm Q) := \sum_{k=1}^K \bm q_k^T \bm A_k \bm q_k
    \qquad \mathrm{s.t.}\quad \bm Q \in \mathrm{St}(d,K),
\end{align}
where $\mathrm{St}(d,K) := \{\bm Q \in \R^{d\times K} : \bm Q^T\bm Q = \bm I_K\}$ denotes the Stiefel manifold, $\bm q_k$ denotes the $k$-th column of $\bm Q$, $\bm Y_l = [\bm y_{l,1},\dots,\bm y_{l,n_l}] \in \R^{d\times n_l}$ for all $l \in [L]$, and
\begin{align}\label{eq:Ak}
    w_{l,k} := \frac{\lambda_k}{\lambda_k+v_l},\quad \bm A_k := \sum_{l=1}^L \frac{w_{l,k}}{n v_l} \bm Y_l \bm Y_l^T. 
\end{align} 
The formulation above assumes that the signal strengths and noise variances are known. Since these parameters are often unknown in practice, we consider a plug-in formulation obtained by replacing $\lambda_k$ and $v_l$ in \eqref{eq:Ak} with estimates
$\widehat\lambda_k$ and $\widehat v_l$, respectively (see Problem \eqref{eq:QPOC plugin}).

Problem \eqref{eq:HePPCA} is a nonconvex optimization problem over the Stiefel manifold. The nonconvexity raises fundamental questions about its optimization properties: does its objective admit spurious local maximizers, and what geometric properties govern the behavior of optimization methods near critical points? Recent studies have made progress in this direction, providing algorithmic and theoretical insights into the optimization properties of HePPCA. Hong et al. \cite{hong2021heppcat} developed alternating   {likelihood} maximization algorithms for jointly estimating the underlying factors and noise variances in HePPCA. Gilman et al. \cite{gilman2025semidefinite} considered the setting where the noise variances for each data group are known, studied a semidefinite relaxation of the HePPCA problem, and established a certificate for global optimality. Wang et al. \cite{wang2023estimation} proposed a generalized power method for solving the same formulation and provided estimation guarantees based on an error-bound property \cite{zhou2017unified}. Despite these advances, a complete geometric understanding of the HePPCA problem remains lacking.

To address this gap, we characterize the global optimization landscape and local geometry of HePPCA with fixed plug-in estimates of the signal strengths and noise variances. Our main contributions are summarized as follows:
\begin{itemize}[leftmargin=*]

\item[$\bullet$]  \textbf{A constructive analysis of the QPOC landscape.} We first show that the population counterpart of HePPCA \eqref{eq:HePPCA} with fixed plug-in estimates is an instance of a   {homogeneous quadratic program with orthogonality constraints} QPOC (see Problem \eqref{eq:QPOC2}). For QPOC, we derive a necessary and sufficient condition for a critical point to be a global maximizer and construct an explicit direction of positive curvature at every nonglobal critical point. This implies that any QPOC admits a benign landscape in the sense that every critical point is either a global maximizer or a strict saddle point (see \Cref{thm:QPOC}). We should point out that the benign-landscape property was previously established in \cite[Corollary 2]{boumal2023brockett}. Our contribution lies in the direct and constructive characterization of critical points and their curvature, which provides the foundation for the subsequent analysis of HePPCA.

\item[$\bullet$] \textbf{Landscape and geometry of population HePPCA.}
Exploiting the QPOC formulation of the population objective, we characterize its critical point set and identify its global maximizers (see \Cref{prop:crit,prop:opt}). By \Cref{thm:QPOC}, the population problem inherits a benign optimization landscape. We further show that the Riemannian Hessian is uniformly
negative definite at global maximizers and has a positive eigenvalue uniformly bounded away from zero at every nonglobal critical point (see \Cref{prop:opt_saddle_G}).

\item[$\bullet$] \textbf{Landscape and geometry of empirical HePPCA.} Building on the population geometry, we establish a deterministic perturbation condition under which the empirical objective has a benign landscape and is locally geodesically strongly concave near every global maximizer (see \Cref{thm:Delta}). Combining this result with sample covariance concentration bounds, we derive a sufficient
sample-size condition ensuring these properties with high probability (see \Cref{thm:theorem5}). Under this sample-size condition, our result applies to fixed plug-in estimates satisfying \Cref{AS:para}, without requiring an explicit bound on their estimation errors. 

\end{itemize}

\subsection{Related Work}

Beyond the aforementioned line of research on HePPCA, a broader literature has explored alternative heteroscedastic PCA settings and extensions of the HePPCA framework. Zhang et al. \cite{zhang2022heteroskedastic} proposed a HeteroPCA algorithm for principal component analysis under heteroscedastic noise across features (as opposed to heteroscedasticity across samples).   {They} proved minimax-rate optimality under a generalized spiked covariance model. Building on this work, Yan et al. \cite{yan2024inference} developed non-asymptotic distributional guarantees for HeteroPCA and demonstrated that it enables accurate inference for the principal subspace, even in the presence of missing data. For heteroscedasticity across samples, further extensions and algorithmic developments include the ALPCAH method,  which jointly estimates a low-rank signal matrix and sample-specific noise variances
using tail singular value regularization \cite{cavazos2025alpcah}; streaming HePPCA, which is designed for online or sequential data with heteroscedastic noise \cite{gilman2023streaming}; and HeMPPCAT, which extends mixtures of probabilistic PCA to accommodate heteroscedastic noise across samples, including samples within the same mixture component \cite{xu2023hemppcat}. 

In nonconvex optimization, a common approach to understanding the behavior of optimization algorithms is through the analysis of the optimization landscape 
\cite{achour2024loss,chen2025complete,chi2019nonconvex,li2019symmetry,sun2018geometric,pmlr-v235-wang24as}. By characterizing the landscape, one can identify different types of critical points, including local minimizers, global minimizers, strict saddle points, and non-strict saddle points (see \Cref{def:crit}). 
In particular, a nonconvex problem is said to admit a {\em benign optimization landscape} if every critical point is either a strict saddle point or a global minimizer.\footnote{This definition is stated for minimization problems. In our setting, we consider a maximization problem, and the definition is adapted accordingly.} Typical examples of such problems include neural collapse problems \cite{yaras2022neural,zhu2021geometric}, low-rank matrix recovery \cite{Bhojanapalli2016,ge2017no}, and group synchronization \cite{mcrae2025nonconvex,mcrae2024benign}. Moreover, the foundational work \cite{lee2019first} showed that first-order methods, including gradient descent, manifold gradient descent, block coordinate descent, and their variants, avoid strict saddle points for almost all initializations. Consequently, for a problem with a benign landscape, randomly initialized first-order methods almost always converge to globally optimal points.

We next review recent advances in the landscape analysis of smooth optimization problems on manifolds.   {Boumal~\cite[Corollary~2]{boumal2023brockett} derives the same benign landscape result for QPOC through an orthogonal-group analysis. Different from their proof, ours works directly on the Stiefel manifold, explicitly characterizing the globally optimal critical points and constructing positive-curvature directions at all others. Boumal~\cite{boumal2024fournorm} also establishes the strict-saddle property for entrywise $\ell_4$-norm maximization over the orthogonal group. Boumal and Criscitiello~\cite{boumal2023lowrankI,boumal2023lowrankII} discuss benign landscapes for low-rank matrix approximation.} Moreover, McRae and Boumal \cite{mcrae2024benign} studied the Burer-Monteiro relaxation of the orthogonal group synchronization problem, which maximizes a quadratic problem over products of Stiefel manifolds. They proved that it has a benign optimization landscape when the relaxation rank is larger than a threshold. Ling \cite{ling2023solving,ling2025local} established
the absence of spurious local optima for related synchronization relaxations under mild conditions. Recently, McRae et al. \cite{mcrae2025nonconvex} derived a sufficient condition under which the Burer-Monteiro relaxation of the $Z_2$-synchronization problem, which lies on the oblique manifold, has a benign optimization landscape under mild conditions. Finally, Yaras et al. \cite{yaras2022neural} established a benign landscape for cross-entropy minimization under a normalized unconstrained-feature model over a product of oblique manifolds.

\subsection{Notation and Organization} We use $\S^n$ to denote the set of all $n\times n$ symmetric matrices. For a positive integer $n$, let $[n] := \{1,\dots,n\}$. For a finite set $\mathcal S$, we use $|\mathcal S|$ to denote its cardinality. Given scalars $a_1,\dots,a_n$, we define $a_{\max} := \max\{a_i: i \in [n]\}$ and $a_{\min} := \min\{a_i: i \in [n]\}$. Given a vector $\bm{a}$, let $\|\bm{a}\|$ denote its Euclidean norm, $a_i$ its $i$-th entry, and $\mathrm{diag}(\bm{a})$ the diagonal matrix with $\bm{a}$ on its diagonal. For a matrix $\bm{A}$, we define $\|\bm{A}\|$ as its spectral norm, $\|\bm{A}\|_F$ as its Frobenius norm, $a_{ij}$ or $\bm A(i,j)$ as its $(i,j)$-th entry, and $\sigma_i(\bm{A})$ as its $i$-th largest singular value. For an integer $n$, let $\mathcal{P}^{n}$ denote the set of all $n \times n$ permutation matrices. 
We associate each $\bm\Pi \in \mathcal{P}^{n}$ with a permutation $\pi : [n] \rightarrow [n]$ through the convention that 
$\bm\Pi(i,j)=1$ if $j=\pi(i)$ and $\bm\Pi(i,j)=0$ otherwise for all $i,j\in[n]$. Under this convention, \(\bm \Pi \bm e_i = \bm e_{\pi^{-1}(i)}\), \(\bm \Pi^T \bm e_i = \bm e_{\pi(i)}\), and, for any diagonal matrix \(\bm D\), \((\bm \Pi^T \bm D\bm \Pi)(i,i)=\bm D(\pi^{-1}(i),\pi^{-1}(i))\).
We use $\mathcal{O}^{n}$ to denote the set of $n \times n$ orthogonal matrices. The symbol $\bm a \circ \bm b$ denotes the Hadamard product of two vectors $\bm a, \bm b \in \R^n$, i.e., $\bm a \circ \bm b=(a_1b_1,\dots,a_nb_n)$. We denote by $\mathrm{BlkD}(\bm{X}_1, \dots, \bm{X}_n)$ the block diagonal matrix with diagonal blocks $\bm{X}_1, \dots, \bm{X}_n$. Blocks with zero rows or zero columns are allowed and are simply omitted from displayed block matrices. In particular, $\bm X_i\in\mathbb R^{m_i\times 0}$ contributes only $m_i$ rows and $\bm X_i\in\mathbb R^{0\times n_i}$ contributes only $n_i$ columns. 

Given $\bm Q \in \mathrm{St}(d,K)$, the tangent space to $\mathrm{St}(d,K)$ at $\bm Q$ is $\mathrm{T}(\bm Q) = \left\{\bm X \in \R^{d\times K}: \mathrm{sym}(\bm Q^T\bm X) = \bm 0  \right\},$ where $\mathrm{sym}(\bm C):= ({\bm C+\bm C^T})/{2}$. We equip $\mathrm{St}(d,K)$ with the Riemannian metric induced by the Frobenius inner product $\langle \bm X,\bm Y\rangle :=\operatorname{Tr}(\bm X^T\bm Y)$. According to \cite[page 162]{boumal2023optimization}, the projection onto the tangent space of $\mathrm{St}(d,K)$ is $\mathrm{Proj}_{\mathrm{T}(\bm Q)}(\bm C)=\bm C-\bm Q\mathrm{sym}(\bm Q^T\bm C)$. Given a smooth function $F:\R^{d \times K} \to \R$, its Riemannian gradient on the Stiefel manifold at $\bm Q$ can be calculated as
\begin{align*}
    \mathrm{grad}\ F(\bm Q) &=  \mathrm{Proj}_{\mathrm{T}(\bm Q)}(\nabla F(\bm Q)). 
\end{align*}
The Riemannian Hessian at $\bm Q$ along the direction $\bm D\in \mathrm{T}(\bm Q)$ can be calculated as 
\begin{align}\label{eq:Re He}
    \mathrm{Hess}\ F(\bm Q)[\bm D] &=\mathrm{Proj}_{\mathrm{T}(\bm Q)}\left(\nabla^2 F(\bm Q)[\bm D] - \bm D \mathrm{sym}(\bm Q^T \nabla F(\bm Q))\right)
\end{align}
We also write $\mathrm{Hess}\ F(\bm Q)[\bm D,\bm D]
:= \left\langle \bm D, \mathrm{Hess}\ F(\bm Q)[\bm D]\right\rangle$.

Now, we formally introduce the definitions of global maximizers, local maximizers, critical points, and strict saddle points for the problem of maximizing a smooth function $H(\bm Q)$ over the Stiefel manifold: 
\begin{align}\label{eq:mani}
    \max_{\bm Q \in \R^{d\times K}} H(\bm Q)\qquad \mathrm{s.t.}\ \bm Q\in \mathrm{St}(d,K). 
\end{align} 
\begin{definition}\label{def:crit}
    For Problem \eqref{eq:mani}, we say that \\
(i) $\bm Q^* \in \mathrm{St}(d,K)$ is a global maximizer if $H(\bm Q^*) \ge H(\bm Q)$ for all $\bm Q \in \mathrm{St}(d,K)$. \\
(ii) $\bm Q^*$ is a local maximizer if there exists a neighborhood $\cal U$ of $\bm Q^*$ such that $H(\bm Q^*) \ge H(\bm Q)$ for all $\bm Q \in \mathcal{U}\cap \mathrm{St}(d,K)$. \\
(iii) $\bm Q^* \in \mathrm{St}(d,K)$ is a critical point if $\mathrm{grad}\ H(\bm Q^*) = \bm 0$. \\
(iv) $\bm Q^* \in \mathrm{St}(d,K)$ is a strict saddle point\footnote{For maximization problems, a direction of positive curvature is an escaping direction \cite{lee2019first}. Note that this definition of strict saddle points therefore includes minima.} if it is a critical point and \\
$\mathrm{Hess}\ H(\bm Q^*)[\bm D,\bm D] > 0$ for a nonzero direction $\bm D\in \mathrm{T}(\bm Q^*)$.  
\end{definition} 

\section{Main Results}\label{sec:main}

We begin by introducing a plug-in formulation of HePPCA   {(see Problem \eqref{eq:HePPCA})} and deriving its population counterpart, which connects HePPCA to quadratic optimization over the Stiefel manifold. In practice, the signal strengths \(\{\lambda_k\}_{k=1}^K\) and noise variances \(\{v_l\}_{l=1}^L\) are typically not known.  We therefore replace $\lambda_k$ and $v_l$ with their plug-in estimates $\widehat\lambda_k$ and $\widehat v_l$, respectively, obtained, for example, using
HePPCAT \cite{hong2021heppcat}. This leads to the following inexact-weight HePPCA objective
\begin{align}\label{eq:QPOC plugin}
    \max_{\bm Q\in\R^{d\times K}}
    \widehat F(\bm Q) := \sum_{k=1}^K \bm q_k^T\widehat{\bm A}_k\bm q_k \qquad \mathrm{s.t.}\quad
    \bm Q\in\mathrm{St}(d,K),
\end{align}
where
\begin{align}\label{eq:Ak plugin}
    \widehat w_{l,k} := \frac{\widehat\lambda_k}{\widehat\lambda_k+\widehat v_l},
    \quad
    \widehat{\bm A}_k := \sum_{l=1}^L \frac{\widehat w_{l,k}}{n\widehat v_l}\bm Y_l\bm Y_l^T.
\end{align}
To facilitate our analysis, we impose the following assumption on the plug-in estimates:
\begin{assumption}\label{AS:para}
The estimates $\{\widehat\lambda_k\}_{k=1}^K$ and $\{\widehat v_l\}_{l=1}^L$ are independent of the data samples $\{\bm y_{l,i}\}$ used to compute $\{\widehat{\bm A}_k\}_{k=1}^K$ and satisfy
\begin{align}\label{eq:w v}
    \widehat{\lambda}_1 > \cdots > \widehat{\lambda}_K > 0,\quad \widehat v_l>0,\ \forall l \in [L].
\end{align}
\end{assumption}
To ensure independence, the estimates $\{\widehat\lambda_k\}_{k=1}^K$
and $\{\widehat v_l\}_{l=1}^L$   {are} computed from a separate sample, independent of the samples used to construct $\{\widehat{\bm A}_k\}_{k=1}^K$. Throughout the subsequent analysis, we condition on these estimates
and treat them as fixed, with all expectations and probability statements understood in this conditional sense. Based on the setup in \Cref{sec:main}, we compute $\E[\widehat{\bm A}_k]$ as follows:  
\begin{lemma}\label{lem:EA}
For each $k \in [K]$, we have
\begin{align}\label{eq:EA}
       \E[\widehat{\bm A}_k] =  \left( \sum_{l=1}^L \frac{n_l\widehat{w}_{l,k}}{n \widehat{v}_l} \right) \bm Q^\star \mathrm{diag}(\bm \lambda) \bm Q^{\star T}  +  \left(\sum_{l=1}^L \frac{\widehat{w}_{l,k} n_lv_l}{n\widehat v_l} \right) \bm I_d. 
\end{align}
\end{lemma}
\begin{proof}
Note that $\widehat w_{l,k}$ and $\widehat v_l$ are fixed. By the independence and zero-mean assumptions on $\bm z_{l,i}$ and $\bm\delta_{l,i}$, we have $\E[\bm y_{l,i}\bm y_{l,i}^T] = \bm Q^\star\diag(\bm\lambda)\bm Q^{\star T} + v_l\bm I_d.$  Consequently, we obtain 
\[
\E[\bm Y_l\bm Y_l^T] = n_l\left(\bm Q^\star\diag(\bm\lambda)\bm Q^{\star T} + v_l\bm I_d\right).
\]
This, together with the definition of $\widehat{\bm A}_k$, yields \eqref{eq:EA}. 
\end{proof}
Replacing each $\widehat{\bm A}_k$ in Problem \eqref{eq:QPOC plugin} with its expectation $\E[\widehat{\bm A}_k]$ in \Cref{lem:EA} and dropping the resulting additive constant, we define $\bm w := (w_1,\dots,w_K)$, where $w_k := \sum_{l=1}^L \frac{n_l \widehat{w}_{l,k}}{n \widehat{v}_l}$. We thus obtain the following population objective of Problem \eqref{eq:QPOC plugin}:
\begin{align}\label{eq:QPOC2}
 \max_{\bm Q \in \mathbb{R}^{d\times K}} 
 G(\bm Q) := \mathrm{Tr}\left( \bm Q^T \bm Q^\star \mathrm{diag}(\bm \lambda) \bm Q^{\star T} \bm Q \mathrm{diag}(\bm w)  \right)
 \qquad\quad \mathrm{s.t.}\quad \bm Q \in \mathrm{St}(d,K).
\end{align} 
This is a special instance of the following quadratic optimization problem with an orthogonality constraint (QPOC) \cite{liu2019quadratic}:
\begin{align}\label{eq:QPOC}
    \max_{\bm Q \in \R^{d \times K}} F(\bm Q) := \mathrm{Tr}\left(\bm Q^T\bm A\bm Q \bm B \right)\qquad \mathrm{s.t.}\ \bm Q \in \mathrm{St}(d,K),
\end{align}
where $\bm A \in \R^{d\times d}$ and $\bm B \in \R^{K\times K}$ are symmetric matrices. Our main results consist of two parts. First, we characterize the global optimization landscape of Problem~\eqref{eq:QPOC}. Second, we apply these results
to the population HePPCA problem~\eqref{eq:QPOC2} and use perturbation analysis to establish analogous properties for its empirical counterpart~\eqref{eq:QPOC plugin}.

\subsection{Optimization Landscape of QPOC}\label{subsec:QPOC result}

For ease of exposition, we denote the critical point set of Problem~\eqref{eq:QPOC} by
\begin{align*}
    \mathcal{Q} := \left\{ \bm Q \in \mathrm{St}(d,K): \mathrm{grad} F(\bm Q) = \bm 0\right\}.
\end{align*}
We next characterize the global optimization landscape of Problem~\eqref{eq:QPOC} by classifying all its critical points.  
\begin{theorem}\label{thm:QPOC}
    Each critical point of Problem \eqref{eq:QPOC} is either a global maximizer or a strict saddle point. 
\end{theorem}

We now make several remarks on the significance of this result. First, this theorem indicates that Problem~\eqref{eq:QPOC} admits a benign optimization landscape in the sense that each critical point is either a global maximizer or a strict saddle point.   {The detailed version is given in \Cref{thm:land}, which further provides a precise criterion that distinguishes global maximizers from strict saddle points among all critical points.} Second, combined with the fact that gradient-based methods with random initialization almost surely escape strict saddle points \cite{lee2019first} and the fact that the \L{}ojasiewicz exponent of Problem~\eqref{eq:QPOC} is $1/2$ \cite{liu2019quadratic}, this result explains why retraction-based gradient methods almost always converge to a global maximizer of Problem~\eqref{eq:QPOC} at a local linear rate. 

Finally, we should point out that the same landscape property has been established in \cite[Corollary 2]{boumal2023brockett}. Nevertheless, our analysis makes a distinct technical contribution through a constructive proof strategy that explicitly connects the structure of critical points to the curvature. This approach provides the foundation for our subsequent analysis of the optimization landscape and local geometry of HePPCA. Specifically, the proof in \cite{boumal2023brockett} uses first- and second-order optimality conditions to show that all second-order critical points on the orthogonal group attain the same objective value, and then transfers the result to the Stiefel manifold by zero-padding the smaller coefficient matrix.  Our technical contribution lies in a direct, constructive analysis on the Stiefel manifold: we derive an explicit criterion for global optimality within the critical point set and construct positive-curvature directions at all remaining critical points   {(see \Cref{remark1} for further discussion)}. This structural characterization provides the foundation for our subsequent local geometry analysis of HePPCA, where we establish strong concavity near global maximizers and uniform positive-curvature bounds at nonglobal critical points. More broadly, the strategy of combining an explicit critical-point characterization with feasible perturbations that reveal curvature may be useful for studying the optimization landscapes and local geometry of other structured problems on manifolds.

\subsection{Optimization Landscape and Geometry of HePPCA}\label{subsec:HPCA result}

We next study the optimization landscape and local geometry of the empirical HePPCA problem with plug-in parameter estimates, namely Problem~\eqref{eq:QPOC plugin}. Throughout this analysis, the plug-in estimates are treated as fixed, and probability statements concern only the observed samples. For ease of exposition, we denote the critical point set of the empirical objective by 
\begin{align*}
    \mathcal{Q}_{\widehat F}
    := \left\{
    \bm Q\in\mathrm{St}(d,K):
    \mathrm{grad}\,\widehat F(\bm Q)=\bm 0
    \right\}.
\end{align*}
 Then, we are ready to show the following theorem, which establishes a benign optimization landscape and local geodesic strong concavity near global maximizers for Problem \eqref{eq:QPOC plugin} under a sufficient sample-size condition.   {Recall that $n_l$ denotes the number of samples in the $l$-th group and $n=\sum_{l=1}^L n_l$ denotes the total number of samples}. 

\begin{theorem}\label{thm:theorem5}
Let $\widehat\beta := \max_{l\in[L]}\frac{\lambda_1+v_l}{\widehat v_l}\sum_{k=1}^K\widehat w_{l,k}$.
Suppose that \Cref{AS:para} holds and 
\begin{align}\label{eq:n}
    n \ge
    \max\left\{\frac{4c_1^2 \widehat{\beta}^2}{\delta_3^2},\frac{2c_1\widehat{\beta}}{\delta_3}\log n \right\}L\left(d\log d+\log n\right), 
\end{align}
where $\delta_3$ is defined in \Cref{thm:Delta} and $c_1$ in \Cref{coro:concen}. Then, with probability at least $1-L/n$, the following statements hold: \\
(i) Every \( \bm{Q} \in \mathcal{Q}_{\widehat F} \) is either a global maximizer or a strict saddle point. \\
(ii) If \( \bm Q  \in \mathcal{Q}_{\widehat F} \) is a global maximizer, then \(\widehat F\) is locally geodesically strongly concave near $\bm Q$. 
\end{theorem}

Now, we make some remarks on the significance of this theorem. First, this theorem establishes that under a sufficient sample-size condition, the HePPCA objective has no spurious local maximizers, and every nonglobal critical point admits a direction of strictly positive curvature. Moreover, local geodesic strong concavity ensures that each global maximizer is isolated and that the objective has strictly negative curvature in every nonzero tangent direction nearby. These properties provide a geometric foundation for first-order methods that escape strict saddles \cite{lee2019first} and for local linear convergence analyses of suitable Riemannian gradient-ascent methods \cite{boumal2023optimization}. 

Second, the proof combines a deterministic perturbation argument with concentration bounds for the sample covariance matrices.   {Building on \Cref{thm:QPOC}, we characterize the critical points of population HePPCA (see \Cref{thm:land}) and establish quantitative curvature bounds at both global maximizers and nonglobal critical points (see \Cref{prop:opt_saddle_G}). We then show that, under sufficiently small coefficient perturbations, the empirical objective inherits the benign landscape and local geodesic strong concavity of its population counterpart (see \Cref{thm:Delta}). Finally, sample covariance concentration (see \Cref{coro:concen}) ensures that the required perturbation bound holds with high probability under the sample-size condition~\eqref{eq:n}, yielding the guarantee stated in \Cref{thm:theorem5}. It is worth pointing out that all auxiliary parameters in the theorem, including $\delta_3$ and $c_1$, are specified in the corresponding results in \Cref{sec:proof}.}

  {Third, the theorem also provides a finite-sample guarantee for HePPCA with known signal strengths and noise variances. In this case, $\widehat\lambda_k=\lambda_k$ and $\widehat v_l=v_l$, so the population weights and perturbation factor reduce to
\begin{align*}
    w_k := \sum_{l=1}^L \frac{n_l}{n}
    \frac{\lambda_k}{v_l(\lambda_k+v_l)},\quad
    \beta := \max_{l\in[L]}
    \frac{\lambda_1+v_l}{v_l}
    \sum_{k=1}^K\frac{\lambda_k}{\lambda_k+v_l}
    \in \left[\frac{\lambda_1}{v_{\min}}, \frac{K\lambda_1}{v_{\min}}\right],\quad v_{\min}:=\min_{l\in[L]}v_l.
\end{align*}
Even when these parameters are known exactly, the sample covariance matrices still fluctuate around their expectations. Condition~\eqref{eq:n} controls the effect of these sampling fluctuations, yielding the stated high-probability guarantee.}

  {Finally, the theorem does not require an explicit accuracy bound on the fixed plug-in estimates. Under \Cref{AS:para}, positive plug-in noise variances and positive, strictly ordered plug-in signal strengths ensure that the HePPCA objective retains a benign global landscape and is locally geodesically strongly concave near every global maximizer. The estimated signal strengths and noise variances nevertheless affect the sufficient sample-size condition~\eqref{eq:n}. For $K\ge2$, closely spaced estimated signal strengths can weaken population curvature (see the proof of \Cref{prop:opt_saddle_G}), while severely underestimated noise variances can amplify sampling fluctuations (see the proof of \Cref{thm:theorem5}). These effects influence the sufficient sample-size condition~\eqref{eq:n} through $\delta_3$ and $\widehat\beta$, potentially making it more restrictive. Overestimating a noise variance reduces the corresponding group's contribution to both sampling fluctuations and population curvature, so its overall effect on the sample-size condition depends on the balance between these two changes.}


\section{Proof of Main Results}\label{sec:proof}

In this section, we first establish general geometric results for QPOC in \Cref{subsec:QPOC}, and then apply these results to HePPCA in \Cref{subsec:HPCA}.

\subsection{Optimization Landscape Analysis of QPOC}\label{subsec:QPOC}

To facilitate our analysis, let $\bm A = \bm U_A \bm \Lambda_A \bm U_A^T $ and $\bm B = \bm U_B \bm \Lambda_B \bm U_B^T $ be spectral decompositions of $\bm A \in \S^d$ and $\bm B \in \S^{K}$, respectively, where $\bm U_A \in \mathcal{O}^d$ and $\bm U_B \in \mathcal{O}^K$. It is straightforward to verify that $ \mathrm{Tr}\left( \bm Q^T \bm A \bm Q \bm B \right) = \mathrm{Tr} \left( \bar{\bm Q}^T \bm \Lambda_A \bar{\bm Q} \bm \Lambda_B \right) $, where $ \bar{\bm Q} = \bm U_A^T \bm Q \bm U_B \in \mathrm{St}(d,K)$. The map $\bm Q\mapsto \bm U_A^T\bm Q\bm U_B$ is a bijective isometry of the Stiefel manifold under the embedded Euclidean metric. Therefore, we assume without loss of generality that
\begin{equation}\label{eq:simpleAB}
\bm A = \mathrm{diag}(a_1, \ldots, a_d),\quad \bm B = \mathrm{diag}(b_1, \ldots, b_K),
\end{equation}
where $ a_1 \geq a_2 \geq \cdots \geq a_d$ and $ b_1 \geq b_2 \geq \cdots \geq b_K$. 
For ease of exposition, we denote the critical point set of Problem~\eqref{eq:QPOC} by
$\mathcal{Q} := \left\{ \bm Q \in \mathrm{St}(d,K): \mathrm{grad} F(\bm Q) = \bm 0\right\}.$
By \cite[Proposition 1]{liu2019quadratic}, this set can be written as
\begin{align}\label{eq:crit set}
    \mathcal{Q} = \left\{ \bm Q \in \mathrm{St}(d,K): \bm A\bm Q \bm B - \bm Q \bm B\bm Q^T\bm A\bm Q = \bm 0\right\}.
\end{align} 

\subsubsection{Characterization of the Critical Point Set}

We first review a key result from \cite{liu2019quadratic} that gives a closed-form characterization of the critical point set of Problem~\eqref{eq:QPOC} when $\bm B$ is nonsingular. We then extend it to the singular case. 

\vspace{-0.1in}
\paragraph{Nonsingular $\bm B$.} We first consider the case where $\bm B$ is nonsingular, i.e., $ b_i \neq 0 $ for each $i \in [K]$. In this case, Liu et al. \cite{liu2019quadratic} derived a closed-form expression for $\mathcal{Q}$. Specifically, let $n_A$ and $ n_B $ be the numbers of distinct eigenvalues of $\bm A$ and $\bm B$, respectively. We assume that $ n_A \geq 2 $, as the objective function is a constant when $n_A = 1$ and all feasible points are trivially optimal.  
Then, there exist indices $ s_0, s_1, \ldots, s_{n_A} $ and $ t_0, t_1, \ldots, t_{n_B} $ that partition the eigenvalues into groups of equal value, such that $ 0 = s_0 < s_1 < \cdots < s_{n_A} = d $, $ 0 = t_0 < t_1 < \cdots < t_{n_B} = K $, and 
 \begin{align*}
&a_{s_0+1} = \cdots = a_{s_1} > a_{s_1+1} = \cdots = a_{s_2} > \cdots > a_{s_{n_A-1}+1} = \cdots = a_{s_{n_A}}, \\ 
&b_{t_0+1} = \cdots = b_{t_1} > b_{t_1+1} = \cdots = b_{t_2} > \cdots > b_{t_{n_B -1}+1}=\cdots=b_{t_{n_B}}.
\end{align*}   
This, together with \eqref{eq:simpleAB}, allows us to express $ \bm A \in \S^d $ and $ \bm B \in \S^K$ as 
\begin{align}
&\bm A = \operatorname{BlkD}\left(a_{s_1} \bm I_{s_1-s_0}, \ldots, a_{s_{n_A}} \bm I_{s_{n_A}-s_{n_A-1}}\right) \in \S^d,\label{eq:block a}\\
&\bm B= \operatorname{BlkD}\left(b_{t_1} \bm I_{t_1-t_0}, \ldots, b_{t_{n_B}} \bm I_{t_{n_B}-t_{n_B-1}}\right) \in \S^K, \label{eq:block b}
\end{align}
respectively. We work with this form as opposed to \eqref{eq:simpleAB} in order to handle rotational invariance that arises when eigenvalues have multiplicity greater than 1. Define
\begin{equation*}
\mathcal{H} := \left\{ (h_1, \ldots, h_{n_A}): \sum_{i=1}^{n_A} h_i = K, \; h_i \in \{0, 1, \ldots, s_i - s_{i-1}\},\ i \in [n_A] \right\}.
\end{equation*}
In addition, given any $ \bm h = (h_1, \ldots, h_{n_A}) \in \mathcal{H} $, define
\begin{align*}
\bm E_i(\bm h) &:= \left[\bm e_{s_{i-1}+1} \cdots \bm e_{s_{i-1}+h_i}\right] \in \mathbb{R}^{d \times h_i} \quad \text{for } i = 1, \ldots, n_A, \nonumber \\
\bm E(\bm h) &:= \left[\bm E_1(\bm h), \cdots, \bm E_{n_A}(\bm h) \right] \in \mathbb{R}^{d \times K},
\end{align*}
where $ \{\bm e_i\}_{i=1}^d $ is the standard basis of $ \mathbb{R}^d $.\footnote{By convention, $h_i = 0$ means that $\bm E_i(\bm h)$ is not included in $\bm E(\bm h)$ for each $i \in [n_A]$.} Intuitively, \(\bm h\) records how many basis vectors are selected from each
eigenspace of \(\bm A\). The matrix \(\bm E(\bm h)\) acts as a selector: for any
compatible matrix \(\bm W\), the product \(\bm W\bm E(\bm h)\) keeps the columns of
\(\bm W\) indexed by these selected basis vectors. 
According to \cite{liu2019quadratic}, the critical point set of Problem~\eqref{eq:QPOC} admits the following closed-form expression.  
\begin{fact}(\cite[Proposition 3]{liu2019quadratic})\label{fact:1}
    Suppose that $\bm A \in \S^d$ and $\bm B \in \S^K$ are given by \eqref{eq:block a} and \eqref{eq:block b}, respectively, and that $b_i \neq 0$ for each $i \in [K]$. Then the critical point set $\mathcal{Q}$ can be expressed as
    \begin{align*}
    \mathcal{Q} = \bigcup_{\bm h \in \mathcal{H}, \bm \Pi \in \mathcal{P}^K} \mathcal{Q}_{\bm h,\bm \Pi},
    \end{align*}
    where every $\bm Q \in \mathcal{Q}_{\bm h,\bm \Pi}$ can be written as
    \begin{equation*}
    \bm Q = \mathrm{BlkD}(\bm U_1, \ldots,\bm U_{n_A})  \bm E(\bm h)  \bm\Pi  \mathrm{BlkD}\left(\bm V_1^T, \ldots, \bm V_{n_B}^T\right)
    \end{equation*}
    for some $\bm U_i \in \mathcal{O}^{s_i - s_{i-1}}$ for each $i \in [n_A]$ and $\bm V_j \in \mathcal{O}^{t_j - t_{j-1}}$ for each $j \in [n_B]$.
\end{fact}

\vspace{-0.1in}
\paragraph{Singular \(\bm B\).} We now extend the above result to the case where \(\bm B\) is singular but nonzero, and derive a closed-form expression for \(\mathcal Q\). With a slight abuse of notation, let $n_B$ denote the number of distinct nonzero eigenvalues of $\bm B$. After reordering the diagonal blocks, we write the zero block last and set $t_{n_B}=\operatorname{rank}(\bm B)$. Similar to the nonsingular case, there exist indices $t_0,t_1,\ldots,t_{n_B},t_{n_B+1}$ and $p \in \{0,1,\dots,n_B\}$ such  that $0=t_0<t_1<\cdots <t_{n_B}<t_{n_B+1}=K$ and 
\begin{equation*}
\begin{aligned}
    & b_{t_0+1}=\cdots=b_{t_1}>\cdots>b_{t_{p-1}+1}=\cdots=b_{t_p}>0,\\
    & 0 > b_{t_p+1}=\cdots=b_{t_{p+1}}>\cdots>b_{t_{n_B-1}+1}=\cdots=b_{t_{n_B}},\\
    &\ b_{t_{n_B}+1}=\cdots=b_{t_{n_B+1}}=0.
\end{aligned}
\end{equation*}
In particular, $p=0$ (resp., $p=n_B$) indicates that there does not exist an index $i$ satisfying $b_i > 0$ (resp., $b_i < 0$). This partition differs from the decreasing-order partition only in that all zero entries are placed in the last block; either sign part may be empty.
Therefore, we have 
\begin{align}\label{eq:block b1}
    \bm B = \mathrm{BlkD}(\widehat{\bm B},\bm 0_{K-t_{n_B}}),\ \ \widehat{\bm B} = 
    \mathrm{BlkD}\left( b_{t_1} \bm I_{t_1-t_0},\dots,b_{t_{n_B}} \bm I_{t_{n_B}-t_{n_B-1}} \right).
    \end{align}
Substituting this form into \eqref{eq:crit set}, together with the argument in \cite[Section 3.1.4]{liu2019quadratic}, yields that the critical point set $\mathcal{Q}$ of Problem~\eqref{eq:QPOC} is equivalent to 
\begin{align}
\label{eq:crit 1}
    \mathcal{Q}
    &=
    \left\{\bm Q=[\bm Q_1,\bm Q_2]\in\mathrm{St}(d,K):
    \bm Q_1\in\mathcal{Q}_1\right\},\\
    \mathcal{Q}_1
    &:=
    \left\{\bm Y\in\mathrm{St}(d,t_{n_B}):
    \bm A\bm Y\widehat{\bm B}
    -\bm Y\widehat{\bm B}\bm Y^T\bm A\bm Y=\bm0
    \right\}.
    \notag
\end{align}
Now, we derive an explicit expression for $\cal Q$ based on \Cref{fact:1}. Define
\begin{align*}
\widetilde{\mathcal{H}} := \left\{ (\bm h_1,\bm h_2):
\begin{array}{l}
\bm h_1=(h_{1,1},h_{1,2},\ldots,h_{1,n_A}),\ 
\bm h_2 = (h_{2,1},h_{2,2},\ldots,h_{2,n_A}), \\
h_{1,i}\in \{0,1,\ldots,s_i-s_{i-1}\},\ i \in [n_A],\
\sum_{i=1}^{n_A} h_{1,i}=t_{n_B}, \\
h_{2,i} = s_i-s_{i-1}-h_{1,i},\ i \in [n_A]
\end{array}
\right\}.
\end{align*}
In addition, given any $\bm h =(\bm h_1,\bm h_2)\in \widetilde{\mathcal{H}}$, define
\begin{align}
&\bar{\bm E}_i(\bm h_1) 
:= \left[\bm e_{s_{i-1}+1} \ \cdots\ \bm e_{s_{i-1}+h_{1,i}}\right] 
\in \mathbb{R}^{d \times h_{1,i}}\ \text{for each}\ i \in [n_A] \notag \\
&\widehat{\bm E}_i(\bm h_2) 
:= \left[\bm e_{s_{i-1}+h_{1,i}+1} \ \cdots\ \bm e_{s_{i}}\right] 
\in \mathbb{R}^{d \times h_{2,i}}\ \text{for each}\ i \in [n_A], \notag\\
& \bar{\bm E}(\bm h_1) := \left[\bar{\bm E}_1(\bm h_1)\ \cdots\ \bar{\bm E}_{n_A}(\bm h_1)\right] \in \R^{d \times t_{n_B}},\label{eq:Eh1}\\
& \widehat{\bm E}(\bm h_2) := \left[\widehat{\bm E}_1(\bm h_2) \ \cdots\ \widehat{\bm E}_{n_A}(\bm h_2)\right] \in \R^{d \times (d-t_{n_B})}, \notag\\
& \bm E(\bm h) 
:= \left[\bar{\bm E}(\bm h_1)\ \widehat{\bm E}(\bm h_2)\right]
\in \mathbb{R}^{d \times d}.
\label{eq:basedef}
\end{align}

\begin{figure}[t]
    \centering
    \makebox[\textwidth][c]{
    \begin{tikzpicture}[
        scale=0.85, transform shape,
        every node/.style={scale=0.85},
        brace/.style={decorate, decoration={calligraphic brace, amplitude=4pt}, thick},
        braceMirror/.style={decorate, decoration={calligraphic brace, amplitude=4pt, mirror}, thick},
        smallbrace/.style={decorate, decoration={calligraphic brace, amplitude=2.5pt}, thick},
        smallbraceMirror/.style={decorate, decoration={calligraphic brace, amplitude=2.5pt, mirror}, thick},
        guideline/.style={densely dotted, gray, thin}
    ]

    \pgfmathsetmacro{\TotalW}{14.0}
    \pgfmathsetmacro{\SplitX}{7.5}

    \pgfmathsetmacro{\TotalH}{5.60}
    \pgfmathsetmacro{\yTop}{5.60}
    \pgfmathsetmacro{\yOne}{4.45}
    \pgfmathsetmacro{\yTwo}{3.55}
    \pgfmathsetmacro{\yThree}{2.40}
    \pgfmathsetmacro{\yFour}{1.15}
    \pgfmathsetmacro{\yBot}{0.0}
    \pgfmathsetmacro{\LabelYOffset}{0.22}

    \pgfmathsetmacro{\wBOne}{2.0}
    \pgfmathsetmacro{\gapBOne}{0.8}
    \pgfmathsetmacro{\wBMid}{2.2}
    \pgfmathsetmacro{\gapBTwo}{0.8}

    \pgfmathsetmacro{\xBEndOne}{\wBOne}
    \pgfmathsetmacro{\xBStartMid}{\xBEndOne + \gapBOne}
    \pgfmathsetmacro{\xBEndMid}{\xBStartMid + \wBMid}
    \pgfmathsetmacro{\xBStartBot}{\xBEndMid + \gapBTwo}
    \pgfmathsetmacro{\xBEndBot}{\SplitX}

    \pgfmathsetmacro{\wOOne}{1.5}
    \pgfmathsetmacro{\gapOOne}{0.8}
    \pgfmathsetmacro{\wOMid}{2.0}
    \pgfmathsetmacro{\gapOTwo}{0.8}

    \pgfmathsetmacro{\xOStartOne}{\SplitX}
    \pgfmathsetmacro{\xOEndOne}{\SplitX + \wOOne}
    \pgfmathsetmacro{\xOStartMid}{\xOEndOne + \gapOOne}
    \pgfmathsetmacro{\xOEndMid}{\xOStartMid + \wOMid}
    \pgfmathsetmacro{\xOStartBot}{\xOEndMid + \gapOTwo}
    \pgfmathsetmacro{\xOEndBot}{\TotalW}


    \draw[fill=cyan!20] (0, \yOne) rectangle (\xBEndOne, \yTop);
    \node at (0.5*\xBEndOne, 0.5*\yOne + 0.5*\yTop + \LabelYOffset)
        {$\bm{I}_{h_{1,1}}$};
    \node at (0.5*\xBEndOne, 0.5*\yOne + 0.5*\yTop - \LabelYOffset)
        {$\bm{0}$};

    \draw[fill=orange!20] (\xOStartOne, \yOne) rectangle (\xOEndOne, \yTop);
    \node at (0.5*\xOStartOne + 0.5*\xOEndOne, 0.5*\yOne + 0.5*\yTop + \LabelYOffset)
        {$\bm{0}$};
    \node at (0.5*\xOStartOne + 0.5*\xOEndOne, 0.5*\yOne + 0.5*\yTop - \LabelYOffset)
        {$\bm{I}_{h_{2,1}}$};

    \draw[fill=cyan!20] (\xBEndOne, \yTwo) rectangle (\xBStartMid, \yOne);
    \draw[fill=orange!20] (\xOEndOne, \yTwo) rectangle (\xOStartMid, \yOne);

    \pgfmathsetmacro{\GapOneX}{(\xBEndOne + \xBStartMid)/2}
    \pgfmathsetmacro{\GapOneY}{(\yOne + \yTwo)/2}
    \node at (\GapOneX, \GapOneY) {\Large $\ddots$};

    \pgfmathsetmacro{\GapOneOX}{(\xOEndOne + \xOStartMid)/2}
    \node at (\GapOneOX, \GapOneY) {\Large $\ddots$};

    \draw[fill=cyan!20] (\xBStartMid, \yThree) rectangle (\xBEndMid, \yTwo);
    \node at (0.5*\xBStartMid + 0.5*\xBEndMid, 0.5*\yThree + 0.5*\yTwo + \LabelYOffset)
        {$\bm{I}_{h_{1,i}}$};
    \node at (0.5*\xBStartMid + 0.5*\xBEndMid, 0.5*\yThree + 0.5*\yTwo - \LabelYOffset)
        {$\bm{0}$};

    \draw[fill=orange!20] (\xOStartMid, \yThree) rectangle (\xOEndMid, \yTwo);
    \node at (0.5*\xOStartMid + 0.5*\xOEndMid, 0.5*\yThree + 0.5*\yTwo + \LabelYOffset)
        {$\bm{0}$};
    \node at (0.5*\xOStartMid + 0.5*\xOEndMid, 0.5*\yThree + 0.5*\yTwo - \LabelYOffset)
        {$\bm{I}_{h_{2,i}}$};

    \draw[smallbraceMirror] (\xBStartMid, \yThree-0.07) -- (\xBEndMid, \yThree-0.07)
        node[midway, below=1pt] {\scriptsize $h_{1,i}$};
    \draw[smallbraceMirror] (\xOStartMid, \yThree-0.07) -- (\xOEndMid, \yThree-0.07)
        node[midway, below=1pt] {\scriptsize $h_{2,i}$};

    \draw[fill=cyan!20] (\xBEndMid, \yFour) rectangle (\xBStartBot, \yThree);
    \draw[fill=orange!20] (\xOEndMid, \yFour) rectangle (\xOStartBot, \yThree);

    \pgfmathsetmacro{\GapTwoX}{(\xBEndMid + \xBStartBot)/2}
    \pgfmathsetmacro{\GapTwoY}{(\yThree + \yFour)/2}
    \node at (\GapTwoX, \GapTwoY) {\Large $\ddots$};

    \pgfmathsetmacro{\GapTwoOX}{(\xOEndMid + \xOStartBot)/2}
    \node at (\GapTwoOX, \GapTwoY) {\Large $\ddots$};

    \draw[fill=cyan!20] (\xBStartBot, \yBot) rectangle (\xBEndBot, \yFour);
    \node at (0.5*\xBStartBot + 0.5*\xBEndBot, 0.5*\yBot + 0.5*\yFour + \LabelYOffset)
        {$\bm{I}_{h_{1,n_A}}$};
    \node at (0.5*\xBStartBot + 0.5*\xBEndBot, 0.5*\yBot + 0.5*\yFour - \LabelYOffset)
        {$\bm{0}$};

    \draw[fill=orange!20] (\xOStartBot, \yBot) rectangle (\xOEndBot, \yFour);
    \node at (0.5*\xOStartBot + 0.5*\xOEndBot, 0.5*\yBot + 0.5*\yFour + \LabelYOffset)
        {$\bm{0}$};
    \node at (0.5*\xOStartBot + 0.5*\xOEndBot, 0.5*\yBot + 0.5*\yFour - \LabelYOffset)
        {$\bm{I}_{h_{2,n_A}}$};

    \draw[guideline] (\xBEndOne, \TotalH) -- (\xBEndOne, 0);
    \draw[guideline] (\xBStartMid, \TotalH) -- (\xBStartMid, 0);
    \draw[guideline] (\xBEndMid, \TotalH) -- (\xBEndMid, 0);
    \draw[guideline] (\xBStartBot, \TotalH) -- (\xBStartBot, 0);

    \draw[guideline] (\xOEndOne, \TotalH) -- (\xOEndOne, 0);
    \draw[guideline] (\xOStartMid, \TotalH) -- (\xOStartMid, 0);
    \draw[guideline] (\xOEndMid, \TotalH) -- (\xOEndMid, 0);
    \draw[guideline] (\xOStartBot, \TotalH) -- (\xOStartBot, 0);

    \draw[thick] (0, 0) rectangle (\TotalW, \TotalH);
    \draw[thick, dashed] (\SplitX, 0) -- (\SplitX, \TotalH);

    \draw[dotted] (0, \yOne) -- (\TotalW, \yOne);
    \draw[dotted] (0, \yTwo) -- (\TotalW, \yTwo);
    \draw[dotted] (0, \yThree) -- (\TotalW, \yThree);
    \draw[dotted] (0, \yFour) -- (\TotalW, \yFour);

    \pgfmathsetmacro{\tickL}{0.07}

    \draw[thick] (0, \yTop) -- ++(-\tickL, 0) node[left, inner sep=2pt] {$1$};
    \draw[thick] (0, \yOne) -- ++(-\tickL, 0) node[left, inner sep=2pt] {$s_1$};
    \node at (-0.32, \GapOneY) {$\vdots$};
    \draw[thick] (0, \yTwo) -- ++(-\tickL, 0) node[left, inner sep=2pt] {$s_{i-1}$};
    \draw[thick] (0, \yThree) -- ++(-\tickL, 0) node[left, inner sep=2pt] {$s_i$};
    \node at (-0.32, \GapTwoY) {$\vdots$};
    \draw[thick] (0, \yFour) -- ++(-\tickL, 0) node[left, inner sep=2pt] {$s_{n_A-1}$};
    \draw[thick] (0, 0) -- ++(-\tickL, 0) node[left, inner sep=2pt] {$d$};

    \draw[thick] (0, \TotalH) -- ++(0, \tickL)
        node[above=1pt, inner sep=1pt] {$1$};

    \draw[thick] (\xBEndOne, \TotalH) -- ++(0, \tickL)
        node[above=1pt, inner sep=1pt, font=\scriptsize] {$h_{1,1}$};

    \draw[thick] (\xBStartMid, \TotalH) -- ++(0, \tickL);
    \draw[thick] (\xBEndMid, \TotalH) -- ++(0, \tickL);
    \draw[thick] (\xBStartBot, \TotalH) -- ++(0, \tickL);

    \draw[thick] (\SplitX, \TotalH) -- ++(0, \tickL)
        node[above=1pt, inner sep=1pt] {$t_{n_B}$};

    \draw[thick] (\xOEndOne, \TotalH) -- ++(0, \tickL)
        node[above=1pt, inner sep=1pt, font=\scriptsize] {$t_{n_B}+h_{2,1}$};

    \draw[thick] (\xOStartMid, \TotalH) -- ++(0, \tickL);
    \draw[thick] (\xOEndMid, \TotalH) -- ++(0, \tickL);
    \draw[thick] (\xOStartBot, \TotalH) -- ++(0, \tickL);

    \draw[thick] (\TotalW, \TotalH) -- ++(0, \tickL)
        node[above=1pt, inner sep=1pt] {$d$};

    \pgfmathsetmacro{\yBraceTop}{\TotalH + 0.42}

    \draw[brace] (0, \yBraceTop) -- (\SplitX, \yBraceTop)
        node[midway, above=3pt] {$\bar{\bm E}(\bm h_1)$};

    \draw[brace] (\SplitX, \yBraceTop) -- (\TotalW, \yBraceTop)
        node[midway, above=3pt] {$\widehat{\bm E}(\bm h_2)$};

    \draw[braceMirror] (0, -0.08) -- (\SplitX, -0.08)
        node[midway, below=1pt] {\scriptsize $t_{n_B}$ columns};
    \draw[braceMirror] (\SplitX, -0.08) -- (\TotalW, -0.08)
        node[midway, below=1pt] {\scriptsize $d-t_{n_B}$ columns};

    \end{tikzpicture}
    }
    \caption{Visualization of the block structure of matrix $\bm E(\bm h)$ defined in \eqref{eq:basedef}. All entries in the unshaded regions are zero.}
    \label{fig:blockplot}
\end{figure}

To illustrate this construction, we provide a schematic illustration of the matrix $\bm E(\bm h)$ in Figure \ref{fig:blockplot}. Based on the above setup, we characterize the critical point set as follows: 
\begin{theorem}\label{thm:singular-case}
Suppose that \(\bm A\in\S^d\) is in the form \eqref{eq:block a} and
\(\bm B\in\S^K\) is singular but nonzero, and is in the form
\eqref{eq:block b1}.
Then the set $\mathcal{Q}$ can be expressed as 
\begin{align*}
    \mathcal{Q} = \bigcup_{\bm h\in \widetilde{\mathcal{H}},\ \bm \Pi\in \mathcal{P}^{t_{n_B}}} \mathcal{Q}_{\bm h,\bm \Pi},
\end{align*}
where every $\bm Q \in \mathcal{Q}_{\bm h,\bm \Pi}$ takes the form 
\begin{align}
    \bm Q = \mathrm{BlkD}(\bm U_1,\ldots,\bm U_{n_A})  \left[\bar{\bm E}(\bm h_1) \bm \Pi\ \widehat{\bm E}(\bm h_2)\bm V_{n_B+1}\right]  \mathrm{BlkD}\left(\bm V_1^T,\ldots,\bm V_{n_B}^T,\bm I_{K-t_{n_B}} \right)
    \label{eq:exp_singular}
\end{align}
for some $\bm U_i\in \mathcal{O}^{s_i-s_{i-1}}$ for each $i\in [n_A]$, $\bm V_j \in \mathcal{O}^{t_j -t_{j-1}}$ for each $j\in[n_B]$, and $\bm V_{n_B+1} \in \mathrm{St}(d-t_{n_B}, K-t_{n_B})$. 
\end{theorem}
\begin{proof}
   Using $\bar{\bm E}_i(\bm h_1)^T\widehat{\bm E}_i(\bm h_2) = \bm 0$ for each $i\in [n_A]$, it is straightforward to verify that every $\bm Q \in \mathcal{Q}_{\bm h,\bm \Pi}$ of the form~\eqref{eq:exp_singular} satisfies $\bm A\bm Q \bm B - \bm Q \bm B\bm Q^T\bm A\bm Q = \bm 0$, and thus $\mathrm{grad} F(\bm Q) = \bm 0$.
    
Using the equivalence between \eqref{eq:crit set} and \eqref{eq:crit 1}, it suffices to take an arbitrary \(\bm Q=[\bm Q_1,\bm Q_2]\in\mathcal Q\) as in \eqref{eq:crit 1} and show that it admits the representation \eqref{eq:exp_singular}. 
We note that $\mathcal{Q}_1$ in \eqref{eq:crit 1} coincides with the critical point set of the following QPOC problem: $    \max_{\bm{Q} \in \mathrm{St}(d, t_{n_B})} \mathrm{Tr}\bigl( \bm{Q}^T \bm{A} \bm{Q} \widehat{\bm{B}} \bigr).$  By \Cref{fact:1}, every $\bm{Q}_1 \in \mathcal{Q}_1$ admits the representation
    \begin{equation}
        \bm{Q}_1 = \mathrm{BlkD}(\bm{U}_1, \ldots, \bm{U}_{n_A}) \bar{\bm{E}}(\bm{h}_1) \, \bm{\Pi}  \mathrm{BlkD}(\bm{V}_1^T, \ldots, \bm{V}_{n_B}^T),\notag
    \end{equation}
    where $\bm{U}_i \in \mathcal{O}^{s_i - s_{i-1}}$ for all $i \in [n_A]$, $\bm{V}_j \in \mathcal{O}^{t_j - t_{j-1}}$ for all $j \in [n_B]$, and $\bar{\bm{E}}(\bm{h}_1)$ is defined in~\eqref{eq:Eh1}. Then  $\bm U := \mathrm{BlkD}(\bm U_1, \ldots, \bm U_{n_A})$ is an orthogonal matrix whose columns form an orthonormal basis for $\mathbb{R}^d$ and $\bar{\bm E}(\bm h_1)$ selects $t_{n_B}$ columns of $\bm U$. 
    
    Moreover, right multiplication by $\bm \Pi$ and $\mathrm{BlkD}(\bm V_1^T,\ldots,\bm V_{n_B}^T)$ does not change the column space of $\bm Q_1$. Hence, with $\mathrm{col}(\cdot)$ denoting the column space, $\mathrm{col}(\bm Q_1)
= \mathrm{col}\bigl(\bm U\bar{\bm E}(\bm h_1)\bigr).$ Since $\bm Q=[\bm Q_1,\bm Q_2]\in \mathrm{St}(d,K)$, we have $\bm Q_1^T\bm Q_2=\bm 0$. 
Therefore,
\(\mathrm{col}(\bm Q_2)\subseteq \mathrm{col}(\bm Q_1)^\perp
=
\mathrm{col}(\bm U\widehat{\bm E}(\bm h_2))\).
Hence, there exists \(\bm V_{n_B+1}\in\mathrm{St}(d-t_{n_B},K-t_{n_B})\) such that  $\bm Q_2 = \bm U\widehat{\bm E}(\bm h_2)\bm V_{n_B+1}.$ Consequently, \(\bm Q\) is in the form \eqref{eq:exp_singular}.
\end{proof}

\subsubsection{Analysis of the Global Optimization Landscape}

Now, we study the optimization landscape of Problem~\eqref{eq:QPOC}. Without loss of generality, we focus on the case where $\bm B$ is singular, since the nonsingular case corresponds to the special case $t_{n_B}=K$, where the zero block is absent. Moreover, for any scalar $\alpha$, replacing $\bm A$ in \eqref{eq:block a} by $\bm A+\alpha\bm I$ changes the objective value on $\mathrm{St}(d,K)$ only by the additive constant $\alpha\mathrm{Tr}(\bm B)$. Therefore, we assume without loss of generality that $\bm A \succ \bm 0$, i.e., $a_i > 0$ for all $i \in [d]$. 
To begin, we identify a condition, expressed in terms of the permutation and the signs of the nonzero diagonal entries of $\bm B$, under which a critical point is a global maximizer of Problem~\eqref{eq:QPOC}. 

Before proceeding with the analysis, we introduce some notation. By \Cref{thm:singular-case}, any $\bm Q \in \mathcal{Q}_{\bm h,\bm \Pi}$ takes the form
\begin{align}\label{eq1:prop saddle}
    \bm Q = \mathrm{BlkD}(\bm U_1,\ldots,\bm U_{n_A}) \left[\bar{\bm E}(\bm h_1)\bm \Pi\ \ \widehat{\bm E}(\bm h_2)\bm V_{n_B+1}\right] \mathrm{BlkD}\left(\bm V_1^T,\ldots,\bm V_{n_B}^T,\bm I\right)
\end{align}
for some $\bm U_i\in \mathcal{O}^{s_i-s_{i-1}}$ for each $i\in[n_A]$, $\bm V_j \in \mathcal{O}^{t_j -t_{j-1}}$ for each $j\in[n_B]$, and $\bm V_{n_B+1} \in \mathrm{St}(d-t_{n_B}, K - t_{n_B})$. For ease of exposition, define
\begin{align}
\label{eq:defuv}
    \bm U:= \mathrm{BlkD}(\bm U_1,\ldots,\bm U_{n_A}),\quad
    \bm V := \mathrm{BlkD}\left(\bm V_1^T,\ldots,\bm V_{n_B}^T,\bm I_{K-t_{n_B}}\right).
\end{align}
Recall that $p$ is the number of positive-eigenvalue blocks of
$\bm B$, so that $t_p$ is the number of its positive eigenvalues. Define 
\begin{align}\label{eq:defPsi}
\bm \Psi := \mathrm{diag}\left(a_1,\ldots,a_{t_p}, a_{d+t_p-t_{n_B}+1}, \ldots, a_d\right) \in \S^{t_{n_B}},  
\end{align}
where the diagonal entries of $\bm \Psi$ are sorted in descending order. In particular, $p = 0$ (resp., $p = n_B$) indicates $\bm \Psi := \mathrm{diag}(a_{d+1-t_{n_B}},\ldots,a_d)$ (resp., $ \bm \Psi := \mathrm{diag}(a_1,\ldots,a_{t_p})$). Further define
\begin{equation}
\label{eq:deftildeA}
 \widetilde{\bm A} := \bm \Pi^T \bar{\bm{E}}(\bm h_1)^T \bm{A} \bar{\bm{E}}(\bm h_1)\bm \Pi \in \R^{t_{n_B} \times t_{n_B}}.
\end{equation}
Since $\bm A$ is diagonal and $\bar{\bm E}(\bm h_1)$ is a coordinate-selection matrix, $\widetilde{\bm A}$ is diagonal and admits the block partition
\[
\widetilde{\bm A} = \mathrm{BlkD}\!\left(\mathrm{diag}(\tilde{\bm a}_1),\dots,\mathrm{diag}(\tilde{\bm a}_{n_B})\right),\ \text{where}\ \tilde{\bm a}_i \in \mathbb{R}^{t_i-t_{i-1}},\ \forall i\in[n_B].
\]
For each $i\in[n_B]$, choose $\bar{\bm \Pi}_i \in \mathcal{P}^{t_i-t_{i-1}}$ such that $\bar{\bm \Pi}_i^T\mathrm{diag}(\tilde{\bm a}_i)\bar{\bm \Pi}_i = \mathrm{diag}\left(\bar{\bm a}_{i} \right),$ where $\bar{\bm a}_i=(\bar a_{i,1},\ldots,\bar a_{i,t_i-t_{i-1}})$ is sorted in decreasing order $\bar a_{i,1}\ge \bar a_{i,2}\ge \cdots\ge \bar a_{i,t_i-t_{i-1}}$ for each $i \in [n_B]$. 
Let $\bar{\bm\Pi}:=\mathrm{BlkD}(\bar{\bm\Pi}_1,\ldots,\bar{\bm\Pi}_{n_B})$ and
\begin{align}\label{eq:Omega_def_rewrite}
\bm\Omega :=\mathrm{BlkD}(\mathrm{diag}\left(\bar{\bm a}_1),\ldots,\mathrm{diag}(\bar{\bm a}_{n_B})\right)= \bar{\bm \Pi}^T\widetilde{\bm A}\bar{\bm \Pi}.
\end{align}
By construction, $\bm\Omega$ is a blockwise sorted version of
$\widetilde{\bm A}$, with the blocks corresponding to the equal-eigenvalue blocks of $\widehat{\bm B}$ defined in \eqref{eq:block b1}. In contrast, $\bm\Psi$ represents the optimal spectral assignment dictated by the signs of the nonzero eigenvalues of $\bm B$: the positive entries of $\widehat{\bm B}$ are paired with the largest diagonal entries of $\bm A$, whereas the negative entries are paired with the
smallest ones. The following proposition shows that if this pattern is realized, then the corresponding critical point is globally optimal.
\begin{proposition}\label{prop:global}
Fix $\bm h\in \widetilde{\mathcal{H}}$ and
$\bm \Pi\in \mathcal{P}^{t_{n_B}}$, and let
$\bm Q \in \mathcal{Q}_{\bm h,\bm \Pi}$. Let $\bm\Psi$ and $\bm\Omega$
be defined as in \eqref{eq:defPsi} and \eqref{eq:Omega_def_rewrite}, respectively.
Suppose that $\bm\Psi=\bm\Omega$. Then $\bm Q$ is a global maximizer.
\end{proposition}
\begin{proof}
By \Cref{thm:singular-case}, any $\bm Q \in \mathcal{Q}_{\bm h,\bm \Pi}$ takes the form \eqref{eq1:prop saddle}. Define \(\bm \alpha:=\mathrm{diag}\left(\bar{\bm E}(\bm h_1)^T\bm A\bar{\bm E}(\bm h_1)\right)\in \mathbb{R}^{t_{n_B}}\). Since $\bm A$ is diagonal and $\bar{\bm E}(\bm h_1)$ is a coordinate-selection matrix, the entries of $\bm\alpha$ are precisely the $t_{n_B}$ diagonal entries of $\bm A$ selected by $\bm h_1$. Substituting \eqref{eq1:prop saddle} into the objective of Problem~\eqref{eq:QPOC} gives
\begin{align}
    F(\bm Q)
    &=
    \mathrm{Tr}\left(
    \bm \Pi^T\bar{\bm E}(\bm h_1)^T
    \bm A
    \bar{\bm E}(\bm h_1)\bm \Pi
    \widehat{\bm B}
    \right) =
    \mathrm{Tr}\left(
    \bm \Pi^T\mathrm{diag}(\bm\alpha)\bm \Pi
    \widehat{\bm B}
    \right).
    \label{eq:plusFX}
\end{align}
Thus, the value of $F(\bm Q)$ depends only on the selected diagonal entries of $\bm A$, determined by $\bm h_1$, and on their ordering, determined by $\bm\Pi$.

We next identify the largest possible value of \eqref{eq:plusFX}. By the rearrangement inequality and the fact that all diagonal entries of $\bm A$ are positive, this inner product is maximized by pairing the positive entries of $\widehat{\bm B}$ with the largest $t_p$ diagonal entries of $\bm A$ and the negative entries with the smallest $t_{n_B}-t_p$ diagonal entries. 
Hence, for any feasible $\bm h_1$ and any permutation $\bm\Pi$, 
\begin{align}
    F(\bm Q)
    \le
    \sum_{i=1}^{t_p} a_i b_i
    +
    \sum_{i=t_p+1}^{t_{n_B}} a_{d-t_{n_B}+i} b_i .
    \label{eq:global_upper_bound}
\end{align}

It remains to show that $\bm Q$ attains this upper bound. Since
$\bar{\bm\Pi}$ only permutes coordinates within the equal-eigenvalue blocks
of $\widehat{\bm B}$, it commutes with $\widehat{\bm B}$. Therefore, by
cyclicity of the trace, (\ref{eq:defPsi}--\ref{eq:Omega_def_rewrite}), and $\bm\Psi=\bm\Omega$, we obtain
\begin{align}\label{eq:FQ}
F(\bm Q)
&=
\mathrm{Tr}\left(
\bm \Pi^T\bar{\bm E}(\bm h_1)^T
\bm A
\bar{\bm E}(\bm h_1)\bm \Pi
\widehat{\bm B}
\right)  =
\mathrm{Tr}\left(
\bar{\bm\Pi}^T
\bm \Pi^T\bar{\bm E}(\bm h_1)^T
\bm A
\bar{\bm E}(\bm h_1)\bm \Pi
\bar{\bm\Pi}
\widehat{\bm B}
\right) \notag\\
&=
\mathrm{Tr}\left(\bm \Psi \widehat{\bm B}\right) =
\sum_{i=1}^{t_p} a_i b_i
+
\sum_{i=t_p+1}^{t_{n_B}} a_{d-t_{n_B}+i} b_i .
\end{align}
Thus $\bm Q$ attains the upper bound in \eqref{eq:global_upper_bound} and is therefore a global maximizer.
\end{proof}

We next show that any critical point $\bm Q \in \mathcal{Q}_{\bm h,\bm\Pi}$
violating $\bm\Psi=\bm\Omega$ is a strict saddle point. The proof is to construct
an ascent direction along which the Riemannian Hessian has
positive curvature. This relies on Givens rotations \cite{givens1958computation}.

\begin{definition}[\bf Givens rotations]\label{def:rota}
Given a positive integer $n$, for any indices $1 \le i \neq  j \le n$ and angle $\theta \in \mathbb{R}$, the Givens rotation matrix $\bm G(i,j,\theta)\in\mathbb{R}^{n\times n}$
is defined as the identity matrix except for the $2\times2$ block indexed by
rows and columns $(i,j)$, where \(g_{ii}=\cos\theta,\ g_{ij}=\sin\theta,\ g_{ji}=-\sin\theta,\ g_{jj}=\cos\theta\).
\end{definition}
The next lemma provides an elementary identity that can be verified by a direct expansion of the affected $(i,j)$ block.
\begin{lemma} \label{lem:givens_effect}
Let $\bm C=\mathrm{diag}(c_1,\dots,c_n)$ and $\bm D=\mathrm{diag}(d_1,\dots,d_n)$ be diagonal matrices. 
Then it holds that 
\[
\mathrm{Tr}\left(\bm G(i,j,\theta)^{T}\bm D\bm G(i,j,\theta)\,\bm C\right)-\mathrm{Tr}(\bm D\bm C)
=(c_j-c_i)(d_i-d_j)\sin^2\theta.
\]
\end{lemma}   
We next present a counting lemma that identifies a useful mismatched coordinate for the subsequent construction of an improving rotation direction. 
\begin{lemma}\label{lem:first_mismatch_counting}
Suppose that $\bm\Omega\neq\bm\Psi$, where $\bm\Psi$ and $\bm\Omega$ are
defined in \eqref{eq:defPsi} and \eqref{eq:Omega_def_rewrite}, respectively.
Write $\omega_i:=\bm\Omega(i,i)$ and $\psi_i:=\bm\Psi(i,i)$ for
$i\in[t_{n_B}]$. If $\omega_i\neq\psi_i$ for some $i\le t_p$, let
$i^*:=\min\{i\le t_p:\omega_i\neq\psi_i\}$; otherwise, let
$i^*:=\max\{i>t_p:\omega_i\neq\psi_i\}$. Then, we have 
\[
    i^*\le t_p \Rightarrow \omega_{i^*}<\psi_{i^*},
    \qquad
    i^*>t_p \Rightarrow \omega_{i^*}>\psi_{i^*}.
\]
Moreover, let $k\in[n_A]$ be such that
$\psi_{i^*}=a_{s_{k-1}+1}=\cdots=a_{s_k}$. If $\omega_j\neq \psi_{i^*}$ for all $j>i^*$
when $i^*\le t_p$, or $\omega_j\neq \psi_{i^*}$ for all $j<i^*$ when $i^*>t_p$,
then $h_{2,k}>0$.
\end{lemma}
\begin{proof}We first consider the case $i^*\le t_p$. By the
definition of $i^*$, $\omega_i=\psi_i$ for all $i<i^*$. Since
$\psi_1,\ldots,\psi_{t_p}$ are the largest $t_p$ diagonal entries of
$\bm A$ in nonincreasing order, all diagonal entries of $\bm A$ that are
strictly larger than $\psi_{i^*}$ appear among $\psi_1,\ldots,\psi_{i^*-1}$.
Hence, if $\omega_{i^*}>\psi_{i^*}$, then the number of entries strictly larger
than $\psi_{i^*}$ in $\omega_1,\ldots,\omega_{i^*}$ would exceed the corresponding
multiplicity in the spectrum of $\bm A$, a contradiction.
Since $\omega_{i^*}\neq \psi_{i^*}$, we obtain
$\omega_{i^*}<\psi_{i^*}$.
The case $i^*>t_p$ is analogous.

It remains to prove the last assertion. We argue by contradiction. Set
$m_k:=s_k-s_{k-1}$. Suppose, to the contrary, that $h_{2,k}=0$. Since
$h_{1,k}+h_{2,k}=m_k$, all $m_k$ copies of the eigenvalue $\psi_{i^*}$ are selected
by $\bar{\bm E}(\bm h_1)$. Hence $\bm\Omega$, which is obtained by sorting
the selected diagonal entries blockwise, must contain exactly $m_k$
diagonal entries equal to $\psi_{i^*}$.

Consider the case $i^*\le t_p$. The diagonal of $\bm\Psi$ contains
at most $m_k$ copies of $\psi_{i^*}$, one of which occurs at
position $i^*$. Since $\omega_i=\psi_i$ for all $i<i^*$, we have $   |\{i<i^*:\omega_i=\psi_{i^*}\}|
    =
    |\{i<i^*:\psi_i=\psi_{i^*}\}|
    \le m_k-1.$ Moreover, $\omega_{i^*}\neq \psi_{i^*}$ by the definition of $i^*$. Hence, if no $j>i^*$ satisfies $\omega_j=\psi_{i^*}$, then the diagonal entries of
$\bm\Omega$ contain at most $m_k-1$ copies of $\psi_{i^*}$, contradicting $h_{2,k}=0$. 
Therefore $h_{2,k}>0$.
The case $i^*>t_p$ is analogous, with the order reversed.
\end{proof}

Based on the above setup, we now construct an ascent direction for each critical point that does not satisfy $\bm\Psi=\bm\Omega$ in the following proposition.
\begin{proposition}\label{prop:saddle}
Let $\bm h\in \widetilde{\mathcal{H}}$, $\bm \Pi\in \mathcal{P}^{t_{n_B}}$,
and $\bm Q \in \mathcal{Q}_{\bm h,\bm \Pi}$ be arbitrary.
Let $\bm\Psi$ and $\bm\Omega$ be defined as in \eqref{eq:defPsi} and
\eqref{eq:Omega_def_rewrite}, respectively. Suppose that $\bm\Psi\ne\bm\Omega$. Then $\bm Q$ is a strict saddle point.
\end{proposition}

\begin{proof}

Let $\pi$ and $\bar{\pi}$ denote the permutations induced by $\bm\Pi$ and
$\bar{\bm\Pi}$, respectively. By assumption, $\bm\Omega\neq\bm\Psi$.
Write $\omega_i:=\bm\Omega(i,i)$ and $\psi_i:=\bm\Psi(i,i)$ for
$i\in[t_{n_B}]$. Choose $i^*$ as in Lemma~\ref{lem:first_mismatch_counting}: if a mismatch
occurs among the first $t_p$ entries, set $ i^*:=\min\{i\le t_p:\omega_i\neq\psi_i\};$ otherwise, set $i^*:=\max\{i>t_p:\omega_i\neq\psi_i\}.$ We distinguish two cases according to whether the value $\psi_{i^*}$ appears on the opposite side of $i^*$ in
$\bm\Omega$.

{\bf Case 1:} Either $i^* \le t_p$ and there exists $j>i^*$ such that
$\omega_j=\psi_{i^*}$, or $i^*>t_p$ and there exists $j<i^*$ such that
$\omega_j=\psi_{i^*}$. If $i^*\le t_p$, then by Lemma~\ref{lem:first_mismatch_counting},
$\omega_{i^*}<\psi_{i^*}=\omega_j$. Since $j>i^*$ and the diagonal entries of $\widehat{\bm B}$ are
nonincreasing, it suffices to rule out the possibility that
$b_{i^*}=b_j$. If $b_{i^*}=b_j$, then $i^*$ and $j$ belong to the same
equal-eigenvalue block of $\widehat{\bm B}$. This contradicts
$\omega_{i^*}<\omega_j$, since $\bm\Omega$ is sorted in nonincreasing order
within each block. Therefore, $b_{i^*}>b_j$.
If $i^*>t_p$, then by Lemma~\ref{lem:first_mismatch_counting},
$\omega_{i^*}>\psi_{i^*}=\omega_j$. A similar argument gives
$b_j>b_{i^*}$. Therefore, in both situations, we have
\begin{equation}\label{eq3:prop_saddle}
    \left( b_j - b_{i^*}\right)\left(\omega_{i^*}-\omega_j\right) > 0 .
\end{equation}


Set $i':=\bar\pi^{-1}(i^*)$ and $j':=\bar\pi^{-1}(j)$. Since
$\bar{\bm\Pi}$ is block diagonal with respect to the equal-eigenvalue
blocks of $\widehat{\bm B}$, we have $b_{i'}=b_{i^*}$ and $b_{j'}=b_j$.
Moreover, by \eqref{eq:Omega_def_rewrite},
\begin{equation}\label{eq2:prop_saddle}
\widetilde{\bm A}(i',i')=\omega_{i^*},
\quad
\widetilde{\bm A}(j',j')=\omega_j .
\end{equation}
Moreover, we define $\bm G(\theta):=\bm G(i',j',\theta)\in\mathbb R^{t_{n_B}\times t_{n_B}}$ and the Stiefel trajectory
\[
\bm Q(\theta) := \bm U\begin{bmatrix}\bar{\bm E}(\bm h_1)\bm \Pi \bm G(\theta)&\widehat{\bm E}(\bm h_2)\bm V_{n_B+1}\end{bmatrix}\bm V.
\]
Obviously, $\bm Q(0) = \bm Q$ due to $\bm G(0) = \bm I$, \eqref{eq1:prop saddle}, and \eqref{eq:defuv}, and $\bm Q(\theta) \in \mathrm{St}(d,K)$. 
Moreover, for a sufficiently small $\theta\neq0$, we obtain
\begin{align*}
F(\bm Q(\theta))-F(\bm Q)
&= \mathrm{Tr}\left(\bm{Q}(\theta)^T{\bm A}\bm{Q}(\theta){\bm B}\right) - \mathrm{Tr}(\bm{Q}^T{\bm A}\bm{Q}{\bm B})\notag \\
& = \mathrm{Tr}\left(\bm{G}(\theta)^T\widetilde{\bm{A}}
\bm{G}(\theta)
\widehat{\bm{B}}\right) - \mathrm{Tr}(\widetilde{\bm{A}}\widehat{\bm{B}}) \notag \\
& =(b_{j'}-b_{i'})\left(\widetilde{\bm A}(i',i')-\widetilde{\bm A}(j',j')\right)\sin^2\theta \notag \\
&= \left( b_j - b_{i^*}\right)\left(\omega_{i^*}-\omega_j\right)\sin^2{\theta} > 0,
\end{align*}
where the second equality uses \eqref{eq:defuv}, \eqref{eq:block b1}, and \eqref{eq:deftildeA}, the third equality uses Lemma~\ref{lem:givens_effect}  with $\bm D=\widetilde{\bm A}$ and $\bm C=\widehat{\bm B}$, the last equality uses \eqref{eq2:prop_saddle}, and the inequality follows from \eqref{eq3:prop_saddle}.
Let $\bm D:=\bm Q'(0)\in \mathrm{T}(\bm Q)$. Since
$\sin^2\theta=\theta^2+o(\theta^2)$ and $\bm Q$ is critical, the standard
second-derivative formula on the Stiefel manifold
\cite[(5.26), p.~107]{boumal2023optimization} yields
$
    \langle \bm D,\mathrm{Hess}\,F(\bm Q)[\bm D]\rangle>0 .
$
Hence $\bm Q$ is a strict saddle point.

{\bf Case 2:} Case~1 does not occur. Choose
$k\in[n_A]$ such that $\psi_{i^*}=a_{s_{k-1}+1}=\cdots=a_{s_k}$. Then there is no
$j>i^*$ with $\omega_j=\psi_{i^*}$ if $i^*\le t_p$, and no $j<i^*$ with
$\omega_j=\psi_{i^*}$ if $i^*>t_p$.  Hence Lemma~\ref{lem:first_mismatch_counting} gives
$h_{2,k}>0$. Choose the last column of $\widehat{\bm E}_k(\bm h_2)$, i.e., let $g_k:=\sum_{i=1}^{k}h_{2,i}$. Then $\widehat{\bm E}(\bm h_2)\bm e_{g_k}$ lies in the $k$-th eigenspace of $\bm A$.


Let $i':=\bar{\pi}^{-1}(i^*)$, so $i'$ and $i^*$ lie in the same block of \eqref{eq:Omega_def_rewrite} and hence $b_{i^*}=b_{i'}$. By \Cref{lem:first_mismatch_counting}, if $i^*\le t_p$ then $b_{i'}>0$ and $\psi_{i^*}>\omega_{i^*}$, while if $i^*>t_p$ then $b_{i'}<0$ and $\psi_{i^*}<\omega_{i^*}$. Then we have
\begin{align}\label{eq:case2_key_sign_rewrite}
b_{i'}\big(\psi_{i^*}-\omega_{i^*}\big)>0.
\end{align}
Set $\bm u:=\bar{\bm E}(\bm h_1)\bm e_{\pi^{-1}(i')}$ and
$\bm v:=\widehat{\bm E}(\bm h_2)\bm e_{g_k}$. Then
\[
\bm u^T\bm A\bm u=\widetilde{\bm A}(i',i')=\omega_{i^*},
\quad
\bm v^T\bm A\bm v=a_{s_k}=\psi_{i^*},
\]
where the identities for $\bm u$ follow from \eqref{eq:deftildeA},
\eqref{eq:Omega_def_rewrite}, and the choice $i'=\bar\pi^{-1}(i^*)$, while
the identity for $\bm v$ follows from the choice of $g_k$.
For $\theta\in\mathbb R$, define $\bar{\bm E}(\bm h_1,\theta)$ and
$\widehat{\bm E}(\bm h_2,\theta)$ by changing only the
$\pi^{-1}(i')$-th column of $\bar{\bm E}(\bm h_1)$ and the $g_k$-th
column of $\widehat{\bm E}(\bm h_2)$ as follows:
\[
\begin{aligned}
\bar{\bm E}(\bm h_1,\theta)_{:,\pi^{-1}(i')}
=
\bm u\cos\theta
+
\bm v\sin\theta,\quad \widehat{\bm E}(\bm h_2,\theta)_{:,g_k} =
-\bm u\sin\theta
+
\bm v\cos\theta,
\end{aligned}
\]
with all other columns unchanged.
Let
\[
\bm Q(\theta) :=
\bm U\begin{bmatrix}
    \bar{\bm E}(\bm h_1,\theta)\bm \Pi & \widehat{\bm E}(\bm h_2,\theta)\bm V_{n_B+1}
\end{bmatrix}\bm V.
\]
By \eqref{eq1:prop saddle} and \eqref{eq:defuv}, $\bm Q(0)=\bm Q$.
Since the two changed columns are obtained by an orthogonal rotation of the
orthonormal pair $\bm u,\bm v$, we also have
$\bm Q(\theta)\in\mathrm{St}(d,K)$. Using \eqref{eq:block b1}, \eqref{eq1:prop saddle}, \eqref{eq:defuv}, 
and the fact that the zero block of $\bm B$ contributes nothing, the objective
function can be written as
\[
F(\bm Q(\theta))
=
\sum_{l=1}^{t_{n_B}} b_l
\left[
\big(\bar{\bm E}(\bm h_1,\theta)\bm e_{\pi^{-1}(l)}\big)^T
\bm A
\big(\bar{\bm E}(\bm h_1,\theta)\bm e_{\pi^{-1}(l)}\big)
\right].
\]
Since only the $\pi^{-1}(i')$-th column of $\bar{\bm E}(\bm h_1)$ is
changed among the selected columns, all terms with $l\ne i'$ cancel in
$F(\bm Q(\theta))-F(\bm Q)$. Hence, for sufficiently small $\theta\ne0$,
\begin{align*}
F(\bm Q(\theta))-F(\bm Q)
&=
b_{i'}
\left[
(\bm u\cos\theta+\bm v\sin\theta)^T
\bm A
(\bm u\cos\theta+\bm v\sin\theta)
-
\bm u^T\bm A\bm u
\right]  \\
&=
b_{i'}
\left[
\bm u^T\bm A\bm u(\cos^2\theta-1)
+
\bm v^T\bm A\bm v\sin^2\theta
\right]  =
b_{i'}(\psi_{i^*}-\omega_{i^*})\sin^2\theta
>0,
\end{align*}
where the second equality uses $\bm u^T\bm A\bm v=0$, and the
last inequality follows from \eqref{eq:case2_key_sign_rewrite}.
As in Case~1, setting $\bm D:=\bm Q'(0)\in\mathrm{T}(\bm Q)$ and using
$\sin^2\theta=\theta^2+o(\theta^2)$, together with the criticality of
$\bm Q$, yields
$    \langle \bm D,\mathrm{Hess}\,F(\bm Q)[\bm D]\rangle>0.$ Thus $\bm Q$ is a strict saddle point.
\end{proof}

  {
\begin{remark}\label{remark1}
\Cref{prop:global,prop:saddle} provide an explicit criterion for global optimality within the critical point set: a critical point is a global maximizer if and only if $\bm\Psi=\bm\Omega$. In contrast to the orthogonal-group reduction in \cite[Corollary~2]{boumal2023brockett}, our proof converts every violation of this criterion into a feasible rotation directly on the Stiefel manifold. Specifically, \Cref{lem:first_mismatch_counting} identifies a spectral mismatch and guarantees that an improving rotation is available. Case~1 rotates two eigen-directions assigned to nonzero weights, whereas Case~2 rotates one such direction toward a missing eigendirection. At every nonglobal critical point $\bm Q$, these constructions yield a feasible curve satisfying $F(\bm Q(\theta))-F(\bm Q)=c\sin^2\theta$ for an explicit coefficient $c>0$. Consequently, $\bm D=\bm Q'(0)$ satisfies $\left\langle \bm D,\mathrm{Hess}\,F(\bm Q)[\bm D]\right\rangle = 2c>0,$ explicitly connecting the spectral mismatch to a positive-curvature direction. 
\end{remark}
}

Building on \Cref{lem:givens_effect,lem:first_mismatch_counting,prop:global,prop:saddle}, we obtain the following detailed version of \Cref{thm:QPOC}:

\begin{theorem}\label{thm:land}
 Let $\bm h\in \widetilde{\mathcal{H}}$, $\bm \Pi\in \mathcal{P}^{t_{n_B}}$, and $\bm Q \in \mathcal{Q}_{\bm h,\bm \Pi}$ be arbitrary. 
Let $\bm\Psi$ and $\bm\Omega$ be defined as in \eqref{eq:defPsi} and
\eqref{eq:Omega_def_rewrite}, respectively. Suppose that $\bm\Psi=\bm\Omega$. Then $\bm Q$ is a global maximizer. 
Otherwise, $\bm Q$ is a strict saddle point. 
\end{theorem}


\subsection{Optimization Landscape and Geometry Analysis of HePPCA}\label{subsec:HPCA}

In this subsection, we apply the landscape characterization of Problem~\eqref{eq:QPOC} to analyze HePPCA. In Section~\ref{subsec:popu}, we exploit the QPOC structure of the population objective to establish its benign optimization landscape and characterize its local geometry. In Section~\ref{subsec:empi}, we show that the empirical objective inherits these properties under suitable sample-size conditions.

\subsubsection{Analysis of the Population HePPCA}\label{subsec:popu}

We now study Problem \eqref{eq:QPOC2}. It follows from \Cref{AS:para} that $\widehat\lambda_1 > \cdots > \widehat\lambda_K$. This, together with \eqref{eq:Ak plugin} and $w_k := \sum_{l=1}^L \frac{n_l \widehat{w}_{l,k}}{n \widehat{v}_l}$, implies that $w_1 > \cdots > w_K$. We denote the critical point set of Problem \eqref{eq:QPOC2} by
\begin{align*}
    \mathcal{Q}_G := \left\{ \bm Q \in \mathrm{St}(d,K): \mathrm{grad}\ G(\bm Q) = \bm 0 \right\}. 
\end{align*}  
Note that Problem \eqref{eq:QPOC2} is a special instance of Problem~\eqref{eq:QPOC}.  Moreover, let \(\bar{\bm Q}^\star\in\mathbb R^{d\times(d-K)}\) be any matrix such that \([\bm Q^\star\ \bar{\bm Q}^\star]\in\mathcal O^d\). Then \(\bm Q^\star \mathrm{diag}(\bm \lambda)\bm Q^{\star T}\) admits the
eigen-decomposition
\[
    \bm Q^\star \mathrm{diag}(\bm \lambda) \bm Q^{\star T}
    =
    \begin{bmatrix}
        \bm Q^\star & \bar{\bm Q}^\star
    \end{bmatrix}
    \begin{bmatrix}
        \mathrm{diag}(\bm \lambda) & \bm 0 \\
        \bm 0 & \bm 0
    \end{bmatrix}
    \begin{bmatrix}
        \bm Q^\star & \bar{\bm Q}^\star
    \end{bmatrix}^T .
\]
Using this decomposition, $\lambda_1 > \cdots > \lambda_K > 0$, $w_1>\cdots>w_K>0$, and \Cref{fact:1}, one can verify that the critical point set of Problem \eqref{eq:QPOC2} can be expressed as follows. Define
\begin{equation*}
    \mathcal{H}_G\! :=\! \left\{\!(h_1,\dots,h_{K+1})\!:\!\! \sum_{i=1}^{K+1} h_i\! =\! K,\ h_i\! \in\! \{0,1\},\ \!\forall i \!\in\! [K],\ h_{K+1}\!\in\! \{0,1,\dots,d\!-\!K\}\! \right\}.
\end{equation*}
In addition, for any $\bm h = (h_1, \ldots, h_K, h_{K+1}) \in \mathcal{H}_G$, we define \(\bm E(\bm h)\in\mathbb R^{d\times K}\) by \(\bm E(\bm h):=[\bm E_1(\bm h)\ \cdots\ \bm E_K(\bm h)\ \bm E_{K+1}(\bm h)]\),
where for each $i \in [K]$,
\begin{align*}
    \bm E_i(\bm h) := 
    \begin{cases}
        \bm e_i, & \text{if } h_i = 1,\\
        \emptyset, & \text{if } h_i = 0,
    \end{cases} \quad 
    \bm E_{K+1}(\bm h) := 
    \begin{cases}
        \begin{bmatrix}
            \bm e_{K+1} & \cdots & \bm e_{K+h_{K+1}} 
        \end{bmatrix}, & \text{if } h_{K+1} \neq 0,\\
        \emptyset, & \text{if } h_{K+1} = 0.
    \end{cases}
\end{align*}
Here, $\emptyset$ denotes an empty matrix of size $d \times 0$, which is omitted from the concatenation, and $\{\bm e_i\}_{i=1}^d \subseteq \mathbb{R}^d$ is the standard basis of $\mathbb{R}^d$. Using \Cref{fact:1} and the setup above, one can verify the following closed-form expression for the critical point set of Problem~\eqref{eq:QPOC2}.
\begin{proposition}\label{prop:crit}
Suppose \Cref{AS:para} holds. Then \(\mathcal{Q}_G = \bigcup_{\bm h \in \mathcal{H}_G,\ \bm \Pi \in \mathcal{P}^K} \mathcal{Q}_{\bm h,\bm \Pi}\),
where every $\bm Q \in \mathcal{Q}_{\bm h,\bm \Pi}$ takes the form
\begin{equation}\label{eq:expresshpi 1}
\bm Q = [\bm Q^\star\ \bar{\bm Q}^\star]  \mathrm{BlkD}\left(\mathrm{diag}(\bm s),\bm S_{K+1} \right)  \bm E(\bm h)  \cdot \bm\Pi \cdot \mathrm{diag}(\bm t)
\end{equation}
for some $\bm s, \bm t \in \R^K$ with $s_i, t_i \in \{1,-1\}$ for each $i\in [K]$ and $\bm S_{K+1} \in \mathcal{O}^{d-K}$. 
\end{proposition}
\begin{proof}
The case $d=K=1$ is immediate, since every feasible point is critical and has the stated form. We therefore consider $d\geq 2$.

Let \(\bm U_\star:=[\bm Q^\star\ \bar{\bm Q}^\star]\in\mathcal O^d\) and
set \(\widetilde{\bm Q}:=\bm U_\star^T\bm Q\). The map \(\bm Q\mapsto \bm U_\star^T\bm Q\) is a smooth isometry of \(\mathrm{St}(d,K)\) onto itself. Hence \(\bm Q\) is a critical point of the population problem if and only if \(\widetilde{\bm Q}\) is a critical point of the diagonal QPOC
\begin{equation}\label{eq:diag_QPOC}
    \max_{\widetilde{\bm Q}\in\mathbb R^{d\times K}}
    \operatorname{Tr}\!\left(
        \widetilde{\bm Q}^T
        \mathrm{BlkD}(\mathrm{diag}(\bm \lambda),\bm 0)
        \widetilde{\bm Q}
        \mathrm{diag}(\bm w)
    \right)\qquad \mathrm{s.t.} \quad 
    \widetilde{\bm Q}\in\mathrm{St}(d,K). 
\end{equation}
By the model assumptions and the weight ordering established above, we have $\lambda_1>\cdots>\lambda_K>0$ and $w_1>\cdots>w_K>0$. Applying \Cref{fact:1} to  \eqref{eq:diag_QPOC} gives
\[
    \widetilde{\bm Q}
    =
    \mathrm{BlkD}\left(\mathrm{diag}(\bm s),\bm S_{K+1}\right)
    \bm E(\bm h)\bm\Pi\mathrm{diag}(\bm t),
\]
where \(\bm h\in\mathcal H_G\), \(\bm\Pi\in\mathcal P^K\),
\(s_i,t_i\in\{1,-1\}\) for \(i\in[K]\), and
\(\bm S_{K+1}\in\mathcal O^{d-K}\).  Substituting this into
\(\bm Q=\bm U_\star\widetilde{\bm Q}\) gives \eqref{eq:expresshpi 1}. The converse inclusion follows
from the same characterization in \Cref{fact:1}. Hence the stated  
representation of \(\mathcal Q_G\) holds.
\end{proof}

We next show that, for fixed $\bm h \in \mathcal{H}_G$ and $\bm \Pi \in \mathcal{P}^K$, all points in $\mathcal{Q}_{\bm h,\bm \Pi}$ have the same type of criticality.
\begin{lemma}\label{lem:equallemma}
Let $\bm h \in \mathcal{H}_G$ and $\bm \Pi \in \mathcal{P}^K$ be arbitrary.   {Then all points in $\mathcal Q_{\bm h,\bm\Pi}$ have the same classification: local minimizers, local maximizers, global minimizers, global maximizers, or strict saddle points.}.
\end{lemma}
\begin{proof}
Let $\bm Q\in\mathcal Q_{\bm h,\bm\Pi}$ be arbitrary. By \Cref{prop:crit}, $\bm Q$ can be written in the form \eqref{eq:expresshpi 1} for some $\bm s, \bm t \in \R^K$ with $s_i,t_i\in\{1,-1\}$ for each $i\in[K]$ and some $\bm S_{K+1}\in \mathcal{O}^{d-K}$. 
Set $ \bm Q' := [\bm Q^\star\ \bar{\bm Q}^\star]\bm E(\bm h)\bm\Pi .$ Define the linear mapping $\psi$ by
\[
    \psi(\bm Z)
    := [\bm Q^\star\ \bar{\bm Q}^\star]
    \mathrm{BlkD}\left(\mathrm{diag}(\bm s),\bm S_{K+1}\right)
    [\bm Q^\star\ \bar{\bm Q}^\star]^T
    \bm Z \mathrm{diag}(\bm t).
\]
Then $\psi(\bm Q')=\bm Q$. Moreover, since the matrices multiplying $\bm Z$ from the left and right are orthogonal, $\psi$ is an isometric bijection from $\mathrm{St}(d,K)$ onto itself. 

Now, we verify that $G(\psi(\bm Z)) = G(\bm Z)$ holds for all $\bm Z$: 
\begin{align*}
    G(\psi(\bm Z)) & = \mathrm{Tr}\left(\psi(\bm Z)^T\bm Q^\star  { \mathrm{diag}(\bm \lambda)} \bm Q^{\star T} \psi(\bm Z) \mathrm{diag}(\bm w)\right)\\
    &=\mathrm{Tr}\left(\psi(\bm Z)^T[\bm Q^\star\ {\bar{\bm Q}^\star}]\blk\left({ \mathrm{diag}(\bm \lambda)},\bm 0\right) [\bm Q^\star\ {\bar{\bm Q}^\star}]^T\psi(\bm Z){\mathrm{diag}(\bm w)}\right)\\
    &=\mathrm{Tr}\left(\mathrm{diag}(\bm t)  \bm Z^T [\bm Q^\star\ {\bar{\bm Q}^\star}] \mathrm{BlkD}\left(\mathrm{diag}(\bm s),\bm S_{K+1}^T\right)\blk\left({ \mathrm{diag}(\bm \lambda)},\bm 0\right)\right.\\
    &\quad\quad\quad~~~\left. \mathrm{BlkD}\left(\mathrm{diag}(\bm s),\bm S_{K+1}\right)[\bm Q^\star\ {\bar{\bm Q}^\star}]^T \bm Z  \mathrm{diag}(\bm t){ \mathrm{diag}(\bm w)}\right)\\
    &=\mathrm{Tr}\left(\bm Z^T[\bm Q^\star\ {\bar{\bm Q}^\star}]\ \blk\left({ \mathrm{diag}(\bm \lambda)},\bm 0\right)\ [\bm Q^\star\ {\bar{\bm Q}^\star}]^T \bm Z  \mathrm{diag}(\bm w)\right)= G(\bm Z),
\end{align*}
where the fourth equality uses $s_i,t_i \in \{1,-1\}$ for each $i \in [K]$. 

Since \(\psi\) is an isometry of \(\mathrm{St}(d,K)\) and \(G\circ\psi=G\),
it preserves local and global optimality. Moreover, since \(\psi\) is a Riemannian isometry and \(G\circ\psi=G\),
the Riemannian Hessian quadratic form is invariant under \(\psi\). Indeed,
for any geodesic \(\gamma\) with \(\gamma(0)=\bm Z\) and
\(\gamma'(0)=\bm D\), the curve \(\psi\circ\gamma\) is a geodesic with
initial velocity \(\mathrm{D}\psi(\bm Z)[\bm D]\), and
\(G(\psi(\gamma(t)))=G(\gamma(t))\). Taking second derivatives at
\(t=0\) gives $\mathrm{Hess}\,G(\psi(\bm Z))\bigl[
\mathrm{D}\psi(\bm Z)[\bm D],\mathrm{D}\psi(\bm Z)[\bm D]
\bigr] = \mathrm{Hess}\,G(\bm Z)[\bm D,\bm D]$. Since $\mathrm{D}\psi(\bm Z)$ maps nonzero tangent directions to
nonzero tangent directions, $\psi$ also preserves the strict-saddle property. Consequently, $\bm Q$ and $\bm Q'$ have the same local and global optimality properties, and one is a strict saddle point if and only if the other is. Since $\bm Q\in\mathcal Q_{\bm h,\bm\Pi}$ was arbitrary, applying this conclusion to $\bm Q=\bm Q_1$ and $\bm Q=\bm Q_2$ proves the result.
\end{proof}

Using this lemma, to determine the type of each critical point in 
$\mathcal{Q}_{\bm h,\bm \Pi}$, it suffices to study the representative 
$\bm Q \in \mathcal{Q}_{\bm h,\bm \Pi}$ of the form
\begin{align}\label{eq:Q}
    \bm Q = \left[\bm Q^\star\ \bar{\bm Q}^\star\right] \cdot \bm E(\bm h) \cdot \bm\Pi.
\end{align}
Using \Cref{thm:land}, we directly obtain the optimization landscape of Problem \eqref{eq:QPOC2} and explicitly characterize its set of global maximizers as follows.
\begin{proposition}\label{prop:opt}
Suppose \Cref{AS:para} holds. Each critical point of Problem \eqref{eq:QPOC2} is either a global maximizer or a strict saddle point. In particular, the set of global maximizers is 
\begin{align}\label{eq:opti}
    \mathcal{Q}_{\mathrm{opt}} := \left\{ \bm Q^\star \mathrm{diag}(\bm s): s_i \in \{1,-1\},\ i \in [K]\right\}.  
\end{align}
\end{proposition}
\begin{proof}
The first statement follows from \Cref{thm:land}. It remains to characterize the global maximizers. By \Cref{lem:equallemma}, it suffices to consider \eqref{eq:Q}. Let $\bm\mu := (\lambda_1,\ldots,\lambda_K,0,\ldots,0)\in\mathbb{R}^d$. For any $\bm h \in \mathcal{H}_G$ and $\bm \Pi \in \mathcal{P}^K$, substituting \eqref{eq:Q} into \eqref{eq:QPOC2} yields
\begin{align*}
    G(\bm Q)
    &=  \mathrm{Tr}\left(\bm \Pi^T \bm{E}^T(\bm h) 
    \mathrm{diag}(\bm\mu) \bm E(\bm h)\bm \Pi \mathrm{diag}(\bm w)\right).
\end{align*}
Since $\lambda_1>\cdots>\lambda_K>0$ and $w_1>\cdots>w_K>0$, the rearrangement inequality implies
$
    G(\bm Q) \le \sum_{k=1}^K \lambda_k w_k,
$
and equality holds for \eqref{eq:Q} if and only if $ \bm E(\bm h)\bm \Pi =
    \begin{bmatrix}
        \bm I_K&        \bm 0
    \end{bmatrix}^T.$
Substituting this condition into \eqref{eq:expresshpi 1}, we obtain $ \bm Q = \bm Q^\star \mathrm{diag}(\bm s)\mathrm{diag}(\bm t).$ Since $s_i,t_i\in\{1,-1\}$ for each $i\in[K]$, this is equivalent to the form in \eqref{eq:opti}. 
\end{proof}

Building on \Cref{thm:land} and \Cref{prop:opt}, we further show that the population objective has a negative-definite Riemannian Hessian at every global maximizer and a positive-curvature direction at every other critical point.
\begin{proposition}\label{prop:opt_saddle_G}
Suppose \Cref{AS:para} holds. The following statements hold for Problem \eqref{eq:QPOC2}: \\
(i) There exists $\delta_1>0$ such that, for every $\bm Q \in \mathcal{Q}_{\mathrm{opt}}$ and every $\bm D \in \mathrm{T}(\bm Q)$,
\begin{align}\label{eq:opt}
    \left\langle\bm D, \mathrm{Hess}\ G(\bm Q)[\bm D]\right\rangle
    \le -\delta_1 \|\bm D\|_F^2.
\end{align}
(ii) 
There exists $\delta_2>0$ such that, for every $\bm Q \in \mathcal{Q}_G \setminus \mathcal{Q}_{\mathrm{opt}}$, there exists a nonzero $\bm D \in \mathrm{T}(\bm Q)$ satisfying
\begin{align}\label{eq:strict-saddle}
    \left\langle\bm D, \mathrm{Hess}\ G(\bm Q)[\bm D]\right\rangle 
    \ge \delta_2 \|\bm D\|_F^2. 
\end{align}
\end{proposition}
\begin{proof}
(i) It suffices to consider $\bm Q=\bm Q^\star$, since the map 
$\bm Q\mapsto \bm Q\mathrm{diag}(\bm s)$ is an isometry on $\mathrm{St}(d,K)$ and leaves $G$ invariant for any $\bm s\in\{1,-1\}^K$.

For any direction $\bm D  \in \mathrm{T}(\bm Q^\star)$, decompose it as
\begin{align*}
    \bm D = \underbrace{(\bm I - \bm Q^\star \bm Q^{\star T})\bm D}_{\bm D_1} + \underbrace{\bm Q^\star \bm Q^{\star T}\bm D}_{\bm D_2}.
\end{align*}
Let $\widehat{\bm D} := \bm Q^{\star T} \bm D$. Since $\bm D\in \mathrm{T}(\bm Q^\star)$, $\widehat{\bm D}$ is skew-symmetric and $\bm D_2=\bm Q^\star\widehat{\bm D}$. We have
\begin{align*}
    \nabla G(\bm Q) &= 2\bm Q^{\star} \mathrm{diag}(\bm \lambda) \bm Q^{\star T} \bm Q  \mathrm{diag}(\bm w),\ \nabla^2 G(\bm Q)[\bm D]=2\bm Q^{\star} \mathrm{diag}(\bm \lambda) \bm Q^{\star T} \bm D \mathrm{diag}(\bm w).
\end{align*}
Therefore, by the formula for the Riemannian Hessian on the Stiefel manifold,
\begin{align*}
    \mathrm{Hess}\ G(\bm Q^\star)[\bm D]
    &= \mathrm{Proj}_{\mathrm{T}(\bm Q^\star)}
    \left(\nabla^2 G(\bm Q^\star)[\bm D]
    - \bm D \mathrm{sym}\left(\bm Q^{\star T} \nabla G(\bm Q^\star)\right)\right) \\
    &=2\left(\bm Q^\star \mathrm{diag}(\bm \lambda)\widehat{\bm D} \mathrm{diag}(\bm w)
    - \bm D\mathrm{diag}(\bm \lambda) \mathrm{diag}(\bm w) \right) \\
    &\quad -2\bm Q^\star\mathrm{sym}\left(\mathrm{diag}(\bm \lambda) \widehat{\bm D}  \mathrm{diag}(\bm w) - \widehat{\bm D}{\mathrm{diag}(\bm \lambda) \mathrm{diag}(\bm w)}
    \right).
\end{align*}
For any matrix $\bm M$, we have $\langle \bm D, \bm Q^\star \mathrm{sym}(\bm M)\rangle
    = \langle \widehat{\bm D}, \mathrm{sym}(\bm M)\rangle = 0,$ where the second equality follows from the fact that the inner product of a symmetric matrix and a skew-symmetric matrix is zero. Hence,
\begin{align}\label{eq:inner_product}
    \left\langle\bm D,\mathrm{Hess}\ G(\bm Q^\star)[\bm D]\right\rangle
    =2\mathrm{Tr}\left(
    \widehat{\bm D}^T{\mathrm{diag}(\bm \lambda)}\widehat{\bm D} \mathrm{diag}(\bm w)
    - (\bm D_1^T \bm D_1 + \widehat{\bm D}^T\widehat{\bm D})\mathrm{diag}(\bm \lambda) \mathrm{diag}(\bm w)\right).
\end{align}
Since $\lambda_1 > \cdots > \lambda_K$ and $w_1 > \cdots > w_K$, we first have
\begin{align}
\label{eq:part1_inequality}
-\mathrm{Tr}\left(\bm{D}_1^T \bm{D}_1 \mathrm{diag}(\bm \lambda) \mathrm{diag}(\bm w)\right)
\le -\lambda_K w_K \| \bm{D}_1 \|_F^2.
\end{align}
Writing $\widehat{\bm D}=(\widehat d_{ij})_{i,j=1}^K$, and using the
skew-symmetry of $\widehat{\bm D}$, we compute
\begin{align}\label{eq:part2_inequality}
\mathrm{Tr}\left(\widehat{\bm{D}}^T \mathrm{diag}(\bm \lambda) \widehat{\bm{D}} \mathrm{diag}(\bm w)
- \widehat{\bm{D}}^T \widehat{\bm{D}} \mathrm{diag}(\bm \lambda) \mathrm{diag}(\bm w)\right) 
& = \sum_{i,j} \left(\lambda_j w_i\widehat{d}_{ji}^2 -  \lambda_i w_i\widehat{d}_{ji}^2 \right) \notag\\
& = -\frac{1}{2}\sum_{i,j}(w_j-w_i)(\lambda_j-\lambda_i)\widehat{d}_{ji}^2 \le 0.
\end{align}
Set
\[
\delta_{\rm rot}:=
\begin{cases}
\displaystyle \left(\min_{i\in[K-1]}(w_i-w_{i+1})\right)
\left(\min_{i\in[K-1]}(\lambda_i-\lambda_{i+1})\right),& K\ge 2,\\[1ex]
+\infty,& K=1.
\end{cases}
\]
For $K=1$, we have $\widehat{\bm D}= \bm 0$. Plugging \eqref{eq:part1_inequality} and \eqref{eq:part2_inequality} into \eqref{eq:inner_product} yields
\begin{align*}
     \left\langle\bm D,\mathrm{Hess}\ G(\bm Q^\star)[\bm D]\right\rangle
    &\le - \sum_{i,j}(w_j-w_i)(\lambda_j-\lambda_i)\widehat{d}_{ji}^2
    - 2\lambda_K w_K\|\bm D_1\|_F^2 \\
    &\le -\min\left\{\delta_{\rm rot},2\lambda_K w_K \right\}\|\bm D\|_F^2,
\end{align*}
where we used the skew-symmetry of $\widehat{\bm D}$, the fact that its diagonal entries are zero, and
$\|\bm D\|_F^2=\|\widehat{\bm D}\|_F^2+\|\bm D_1\|_F^2$. This proves \eqref{eq:opt} with $\delta_1=\min\{\delta_{\rm rot},2\lambda_Kw_K\}$.

(ii) If \(\mathcal{Q}_G\setminus\mathcal{Q}_{\mathrm{opt}}=\emptyset\), the claim is vacuous and we set \(\delta_2=1\). Otherwise, define
\[
\rho(\bm Q):=
\max_{\bm D\in \mathrm{T}(\bm Q),\ \|\bm D\|_F=1}
\left\langle
\bm D,\mathrm{Hess}\,G(\bm Q)[\bm D]
\right\rangle .
\]
The maximum is attained by compactness of the unit sphere in \(\mathrm{T}(\bm Q)\). By \Cref{prop:saddle}, \(\rho(\bm Q)>0\) on \(\mathcal{Q}_G\setminus\mathcal{Q}_{\mathrm{opt}}\). This set is compact by \Cref{prop:crit,lem:equallemma,prop:opt}, and \(\rho\) is continuous because the Riemannian Hessian depends continuously on \(\bm Q\). Hence
$
\delta_2:=
\frac12
\min_{\bm Q\in\mathcal{Q}_G\setminus\mathcal{Q}_{\mathrm{opt}}}
\rho(\bm Q)>0.
$
For each \(\bm Q\in\mathcal{Q}_G\setminus\mathcal{Q}_{\mathrm{opt}}\), choosing a unit maximizer in the definition of \(\rho(\bm Q)\) gives \eqref{eq:strict-saddle}.
\end{proof}

\subsubsection{Analysis of the Empirical HePPCA}\label{subsec:empi}

In this subsection, we analyze the empirical HePPCA problem
with plug-in parameter estimates, namely
Problem~\eqref{eq:QPOC plugin}. 
Although \Cref{prop:crit} provides an explicit characterization
of the critical point set of the population objective $G$,
an analogous characterization is generally unavailable for
the empirical objective $\widehat F$. To quantify the deviations
of the empirical coefficient matrices from their expectations,
define
\begin{align}\label{eq:Delta}
    \bm\Delta_k
    := \widehat{\bm A}_k-\E[\widehat{\bm A}_k],
    \quad k\in[K],
\end{align}
where $\widehat{\bm A}_k$ is defined in \eqref{eq:Ak plugin}. We begin with the following global error bound, which follows from \cite[Corollary~1]{liu2019quadratic}. It controls the distance from any feasible point to
$\mathcal Q_G$ in terms of the norm of the Riemannian
gradient of $G$ at that point.
\begin{lemma} 
\label{lem:plg}
There exists a constant $\kappa_1 > 0$ such that for all $\bm Q \in \mathrm{St}(d,K)$,
\begin{align*}
    \mathrm{dist}(\bm Q, \mathcal{Q}_G) \le \kappa_1 \|\mathrm{grad}\ G(\bm Q)\|_F.
\end{align*}
\end{lemma}
If $\mathcal{Q}_G\setminus\mathcal{Q}_{\mathrm{opt}}=\emptyset$, we set $\mathrm{dist}(\bm Q,\emptyset)=+\infty$ so the corresponding neighborhood assertions are vacuous. We next apply this bound to the critical points of \(\widehat F\), using the following perturbation identity. By \eqref{eq:QPOC plugin}, \eqref{eq:EA}, \eqref{eq:QPOC2}, and \eqref{eq:Delta}, for any \(\bm Q\in\mathrm{St}(d,K)\),
\begin{align}\label{eq:F-G}
   \widehat F(\bm Q) = G(\bm Q) + \sum_{k=1}^K
    \mathrm{Tr}\!\left(\bm Q^T\bm\Delta_k\bm Q\bm E_{kk}\right)
    + \sum_{k=1}^K\sum_{l=1}^L\frac{\widehat w_{l,k}n_lv_l}{n\widehat v_l},
\end{align}
where \(\bm E_{kk}:=\bm e_k\bm e_k^T \in \mathbb{R}^{K \times K}\). This allows us to bound the distance from any critical point of $\widehat{F}$ to the critical point set $\mathcal{Q}_G$.

\begin{lemma}\label{lem:dist_twosets}
For each \( \bm Q \in \mathcal{Q}_{\widehat F} \), we have
\begin{align}
\label{eq:distpl}
\mathrm{dist}(\bm Q, \mathcal{Q}_G) \le {2\kappa_1\sum_{k=1}^K \|\bm\Delta_k\|},
\end{align}
where \( \kappa_1 \) is defined in \Cref{lem:plg}.
\end{lemma}
\begin{proof}
Let $\bm Q\in\mathcal Q_{\widehat F}$ be arbitrary. By \eqref{eq:F-G}, we have
\begin{align*}
    \bm 0 = \mathrm{grad}\ \widehat F(\bm Q)
    = \mathrm{grad}\ G(\bm Q)
    + \sum_{k=1}^K \mathrm{grad}\,\mathrm{Tr}\left(\bm Q^T\bm \Delta_k \bm Q \bm E_{kk}\right).
\end{align*}
Since \(\bm\Delta_k\) and \(\bm E_{kk}\) are symmetric, we have
\begin{align*}
    \|\mathrm{grad}\ G(\bm Q)\|_F
    &= \left\| \sum_{k=1}^K \mathrm{grad}\,\mathrm{Tr}\left(\bm Q^T\bm \Delta_k \bm Q \bm E_{kk}\right) \right\|_F = 2\left\|\sum_{k=1}^K \mathrm{Proj}_{\mathrm{T}(\bm Q)}
    \left( \bm \Delta_k \bm Q \bm E_{kk}\right)\right\|_F \\
    &\le 2\sum_{k=1}^K \left\|\bm \Delta_k \bm Q \bm E_{kk}\right\|_F
    = 2\sum_{k=1}^K \left\|\bm \Delta_k \bm q_k\right\|
    \le 2\sum_{k=1}^K \left\|\bm \Delta_k\right\|,
\end{align*}
where \(\bm q_k\) is the \(k\)-th column of \(\bm Q\). Together with \Cref{lem:plg}, this gives \eqref{eq:distpl}.
\end{proof}

Based on the above results, we now perturb the landscape of Problem \eqref{eq:QPOC2} to analyze that of Problem \eqref{eq:QPOC plugin} when $\widehat{\bm A}_k$ is close to $\E[\widehat{\bm A}_k]$.
\begin{proposition}\label{prop:F_classify}
Suppose that \Cref{AS:para} holds and 
\begin{align}\label{eq:delta_condition}
\sum_{k=1}^K \| \bm \Delta_k \| \le \frac{\min\{\delta_1,\delta_2\}}{12},
\end{align}
where $\delta_1$ and $\delta_2$ are defined in \Cref{prop:opt_saddle_G}. Then the following statements hold: \\
(i) There exists \( \gamma_1 > 0\) such that, for every $\bm Q \in \mathrm{St}(d,K)$ satisfying
\[
\dist(\bm Q, \mathcal Q_{\rm opt})\le \gamma_1,
\]
we have
\begin{align}\label{eq1:prop F class}
    \langle\bm D, \mathrm{Hess}\ \widehat F(\bm Q)[\bm D]\rangle\le -\frac{\delta_1}{3} \|\bm D\|_F^2,\quad \forall \bm D \in \mathrm{T}(\bm Q),
\end{align}
where $\mathcal Q_{\rm opt}$ is defined in \Cref{prop:opt}.\\
(ii) There exists \( \gamma_2 > 0\) such that, for every $\bm Q \in \mathrm{St}(d,K)$ satisfying $\mathrm{dist}(\bm Q, \mathcal{Q}_G\setminus\mathcal{Q}_{\mathrm{opt}}) \le \gamma_2$, there exists a nonzero $\bm D \in \mathrm{T}(\bm Q)$ such that
\begin{align}\label{eq2:prop F class}
    \langle \bm D, \mathrm{Hess}\ \widehat F(\bm Q)[\bm D]\rangle
    \ge \frac{\delta_2}{3} \|\bm D\|_F^2.
\end{align}
\end{proposition}
\begin{proof}
By \eqref{eq:F-G}, we have 
\begin{align}\label{eq1:prop F}
    \mathrm{Hess}\ \widehat F(\bm Q) = \mathrm{Hess}\ G(\bm Q)
    +\sum_{k=1}^K \mathrm{Hess}\, \mathrm{Tr}(\bm Q^T\bm \Delta_k\bm Q \bm E_{kk}).
\end{align}
For each $\bm D \in \mathrm{T}(\bm Q)$, we have
\begin{align*}
\nabla \mathrm{Tr}(\bm Q^T\bm \Delta_k\bm Q\bm E_{kk})
= 2\bm \Delta_k\bm Q\bm E_{kk},\quad
\nabla^2 \mathrm{Tr}(\bm Q^T\bm \Delta_k\bm Q\bm E_{kk})[\bm D]
 = 2 \bm \Delta_k\bm D \bm E_{kk},
\end{align*}
where $\bm\Delta_k$ and $\bm E_{kk}$ are symmetric. By \eqref{eq:Re He}, nonexpansiveness of the tangent projection, $\|\bm\Delta_k\bm D\bm E_{kk}\|_F\le\|\bm\Delta_k\|\,\|\bm D\|_F$, and $\|\mathrm{sym}(\bm Q^T\bm\Delta_k\bm Q\bm E_{kk})\|_F\le \|\bm\Delta_k\|$, we obtain
$ 
   \left|\left\langle\bm D, \mathrm{Hess}\,
   \mathrm{Tr}(\bm Q^T\bm \Delta_k\bm Q \bm E_{kk})[\bm D]\right\rangle \right|
   \le 4\|\bm \Delta_k\|\|\bm D\|_F^2.
$
This, together with \eqref{eq1:prop F}, implies
\begin{equation}
\label{eq:hessian_def}
\begin{aligned}
&\left|\mathrm{Hess}(\widehat F\!-\!G)(\bm Q)\![\bm D,\bm D]\right|\le 4\|\bm D\|_F^2\!\sum_{k=1}^K\|\bm \Delta_k\|.
\end{aligned}
\end{equation}

(i) Using \Cref{prop:opt_saddle_G}(i), compactness of \( \mathcal{Q}_{\mathrm{opt}} \), and smoothness of \(G\), there exists \( \gamma_1 > 0 \) such that for any \( \widetilde{\bm Q} \) in a neighborhood of \( \mathcal{Q}_{\mathrm{opt}} \), i.e., $\widetilde{\bm Q} \in \left\{ \bm Q \in \mathrm{St}(d, K) : \mathrm{dist}(\bm Q, \mathcal{Q}_{\mathrm{opt}}) \le \gamma_1 \right\},$ the following inequality holds:
\begin{align}\label{eq2:prop F}
  \langle \widetilde{\bm D}, \mathrm{Hess}\ G(\widetilde{\bm Q})[\widetilde{\bm D}] \rangle
  \le -\frac{2}{3} \delta_1 \| \widetilde{\bm D} \|_F^2,
  \quad \forall \widetilde{\bm D} \in \mathrm{T}(\widetilde{\bm Q}).  
\end{align}
For each $\bm Q\in \mathrm{St}(d,K)$ such that $\mathrm{dist}(\bm Q, \mathcal{Q}_{\mathrm{opt}}) \le \gamma_1$, we have for all $\bm D \in \mathrm{T}(\bm Q)$,
\begin{align*}
    \langle \bm D, \mathrm{Hess}\ \widehat F(\bm Q)[\bm D]\rangle
    &\le 4\sum_{k=1}^K \|\bm \Delta_k\|\cdot \|\bm D\|_F^2
    + \langle  \bm D, \mathrm{Hess}\ G(\bm Q)[\bm D]\rangle \le  -\frac{1}{3}\delta_1 \|\bm D\|_F^2,
\end{align*}
where the first inequality follows from \eqref{eq:hessian_def}, and the last inequality uses \eqref{eq:delta_condition} and \eqref{eq2:prop F}.

(ii) If \(\mathcal{Q}_G\setminus\mathcal{Q}_{\mathrm{opt}}=\emptyset\), fix any \(\gamma_2>0\). Otherwise, set \(\mathcal S:=\mathcal{Q}_G\setminus\mathcal{Q}_{\mathrm{opt}}\). By \Cref{prop:opt_saddle_G}(ii), for each \(\bm Q_0\in\mathcal S\), there is a unit direction \(\bm D_0\in \mathrm{T}(\bm Q_0)\) with \(\langle \bm D_0,\mathrm{Hess}\,G(\bm Q_0)[\bm D_0]\rangle\ge \delta_2\). For \(\bm Q\in\mathrm{St}(d,K)\), define \(\bm D_{\bm Q_0}(\bm Q):=\mathrm{Proj}_{\mathrm{T}(\bm Q)}(\bm D_0)\). Since \(\bm D_{\bm Q_0}(\bm Q_0)=\bm D_0\neq \bm 0\), from continuity of the tangent projection and the Riemannian Hessian, it follows that there exists \(r_{\bm Q_0}>0\) such that \(\bm D_{\bm Q_0}(\bm Q)\neq \bm 0\) and \[
\frac{\left\langle
\bm D_{\bm Q_0}(\bm Q),
\mathrm{Hess}\,G(\bm Q)[\bm D_{\bm Q_0}(\bm Q)]
\right\rangle}
{\|\bm D_{\bm Q_0}(\bm Q)\|_F^2}
\ge
\frac{2\delta_2}{3},\quad \forall \bm Q\in U_{\bm Q_0},
\]
where \(U_{\bm Q_0}:=\{\bm Q\in\mathrm{St}(d,K):\|\bm Q-\bm Q_0\|_F<r_{\bm Q_0}\}\). The sets $\{U_{\bm Q_0}: \bm Q_0 \in \mathcal S\}$ form an open cover of compact \(\mathcal S\), so there exist \(\bm Q_1,\ldots,\bm Q_m\in\mathcal S\) such that \(\mathcal S \subseteq \bigcup_{i=1}^m U_{\bm Q_i}\). Set \(U:=\bigcup_{i=1}^m U_{\bm Q_i}\). Since \(U\) is an open neighborhood of compact \(\mathcal S\) in \(\mathrm{St}(d,K)\), choose \(\gamma_2>0\) such that
\[
\left\{\bm Q\in\mathrm{St}(d,K):\mathrm{dist}(\bm Q,\mathcal S)\le \gamma_2\right\} \subseteq U.
\]
Thus, if \(\mathrm{dist}(\bm Q,\mathcal S)\le\gamma_2\), then \(\bm Q\in U_{\bm Q_i}\) for some \(i\). Taking \(\bm D:=\bm D_{\bm Q_i}(\bm Q)=\mathrm{Proj}_{\mathrm{T}(\bm Q)}(\bm D_0^i)\), where \(\bm D_0^i\in\mathrm{T}(\bm Q_i)\) is the corresponding unit direction, gives \(\bm D\neq\bm0\) and \(\langle \bm D,\mathrm{Hess}\,G(\bm Q)[\bm D]\rangle\ge(2\delta_2/3)\|\bm D\|_F^2\). Combining this with \eqref{eq:delta_condition} and \eqref{eq:hessian_def} gives \eqref{eq2:prop F class}.
\end{proof}

To proceed, we introduce the following definitions. Let \(\mathcal B_{\rm geo}(\bm Q_0,\eta)\) denote a geodesic ball in \(\mathrm{St}(d,K)\)  centered at $\bm Q_0$ with radius $\eta$ with respect to the Riemannian metric. A subset \(S\subseteq\mathrm{St}(d,K)\) is called geodesically
convex if any two points in \(S\) can be joined by a geodesic segment lying
entirely in \(S\) (see \cite[Chapter~11]{boumal2023optimization}). 

\begin{lemma}\label{lem:local_unique_critical}
Let \(H:\mathrm{St}(d,K)\to\mathbb R\) be a smooth function. Let \(\eta>0\)
be such that the geodesic ball
\(\mathcal B_{\rm geo}(\bm Q_0,\eta)\) is  geodesically convex. Suppose that there exists \(\mu>0\) such that 
\[
    \left\langle \bm D,\mathrm{Hess}\,H(\bm Z)[\bm D]\right\rangle
    \le -\mu\|\bm D\|_F^2,
    \quad
    \forall \bm Z\in\mathcal B_{\rm geo}(\bm Q_0,\eta),\
    \bm D\in\mathrm{T}(\bm Z).
\]
Then \(\mathcal B_{\rm geo}(\bm Q_0,\eta)\) contains at most one critical point
of \(H\).
\end{lemma}
\begin{proof}
Suppose that two distinct critical points $\bm Q_a,\bm Q_b$ of $H$ lie
in $\mathcal B_{\rm geo}(\bm Q_0,\eta)$. By geodesic convexity, there exists a nonconstant and affinely parameterized geodesic $\gamma:[0,1]\to\mathcal B_{\rm geo}(\bm Q_0,\eta)$ such that $\gamma(0)=\bm Q_a$ and $\gamma(1)=\bm Q_b$. A nonconstant
geodesic has constant positive speed; see
\cite[Exercise~5.40]{boumal2023optimization}. Indeed, 
\[
\frac{\mathrm d}{\mathrm dt}\|\dot\gamma(t)\|^2 = 2\left\langle
\nabla_{\dot\gamma(t)}\dot\gamma(t),\dot\gamma(t)\right\rangle =0,
\]
where the last equality follows from the geodesic equation. Since $\gamma$ is nonconstant,
$\|\dot\gamma(t)\|_F>0$. 

Let $\varphi:=H\circ\gamma$. Since $\gamma$ is a geodesic, the Hessian
bound gives 
\[
    \varphi''(t)
    =
    \left\langle \dot\gamma(t),
    \mathrm{Hess}\,H(\gamma(t))[\dot\gamma(t)]\right\rangle
    \le -\mu\|\dot\gamma(t)\|_F^2 <0.
\]
Thus, $\varphi'$ is strictly decreasing.
However, the criticality of $\bm Q_a$ and $\bm Q_b$ implies
$\varphi'(0)=\varphi'(1)=0$, which is a contradiction.
\end{proof}

With the plug-in parameter estimates treated as fixed,
we next choose several constants used below. Since $\mathcal Q_{\mathrm{opt}}$ is finite, the existence of sufficiently small geodesically convex balls, \Cref{prop:opt_saddle_G}(i), and continuity of the Riemannian Hessian
of $G$ imply that there exists a common $\eta>0$ such that, for every $\bm Q_1\in\mathcal Q_{\mathrm{opt}}$, the geodesic ball $\mathcal B_{\rm geo}(\bm Q_1,\eta)$ is geodesically convex, is contained in the $\gamma_1$-neighborhood of $\mathcal Q_{\mathrm{opt}}$, and satisfies
\[
    \mathrm{Hess}\,G(\bm Z)[\bm D,\bm D]
    \le
    -\frac{\delta_1}{2}\|\bm D\|_F^2,
    \quad
    \forall \bm Z\in\mathcal B_{\rm geo}(\bm Q_1,\eta),\
    \bm D\in\mathrm T(\bm Z).
\]
Moreover, since each $\mathcal B_{\rm geo}(\bm Q_1,\eta)$ is an open neighborhood of $\bm Q_1$ in $\mathrm{St}(d,K)$ and $\mathcal Q_{\mathrm{opt}}$ is finite, there exists $\tau>0$ such that
\[
    \mathbb B(\bm Q_1,\tau)\cap\mathrm{St}(d,K)
    \subseteq
    \mathcal B_{\rm geo}(\bm Q_1,\eta),
    \quad
    \forall \bm Q_1\in\mathcal Q_{\mathrm{opt}},
\]
where $\mathbb B(\bm Q_1,\tau) := \left\{ \bm Q: \|\bm Q - \bm Q_1\|_F \le \tau \right\}$.

Set
\begin{align}\label{eq:sigma0}
\sigma_0:=
\begin{cases}
\operatorname{dist}\!\left(\mathcal Q_{\mathrm{opt}},
\mathcal Q_G\setminus\mathcal Q_{\mathrm{opt}}\right),
&\mathcal Q_G\setminus\mathcal Q_{\mathrm{opt}}\neq\emptyset,\\
1,&\mathcal Q_G\setminus\mathcal Q_{\mathrm{opt}}=\emptyset,
\end{cases}    
\end{align}
Now, we prove that \(\sigma_0>0\). If
\(\mathcal Q_G\setminus\mathcal Q_{\mathrm{opt}}=\emptyset\), this follows
directly from the definition. Otherwise, fix any
\(\bm Q_1\in\mathcal Q_{\mathrm{opt}}\). By the choice of \(\eta\), the ball
\(\mathcal B_{\rm geo}(\bm Q_1,\eta)\) is geodesically convex and
\(\mathrm{Hess}\,G\) is negative definite on this ball. Applying \Cref{lem:local_unique_critical} with \(H=G\), we conclude that each \(\bm Q_1\in\mathcal Q_{\mathrm{opt}}\) is isolated in \(\mathcal Q_G\). Since \(\mathcal Q_{\mathrm{opt}}\) is finite and \(\mathcal Q_G\) is compact, it follows that $\sigma_0>0$. We are now ready to state the perturbation result for Problem \eqref{eq:QPOC plugin}.

\begin{theorem}\label{thm:Delta}
Suppose that \Cref{AS:para} holds and \( \{\bm \Delta_k\}_{k=1}^K \) satisfies
\begin{align}\label{eq:final_condition}
\sum_{k=1}^K \| \bm \Delta_k \| \le \delta_3 := \min \left\{\frac{\min\{\gamma_1,\gamma_2,\tau,\sigma_0/3,1/2\}}{2\kappa_1}, \frac{\min\{\delta_1,\delta_2\}}{12}\right\},
\end{align}
where $\gamma_1,\gamma_2$ are defined in \Cref{prop:F_classify},
$\delta_1,\delta_2$ are defined in \Cref{prop:opt_saddle_G}, $\tau,\sigma_0$ are fixed above, and $\kappa_1$ is defined in
\Cref{lem:plg}. Then every \( \bm{Q} \in \mathcal{Q}_{\widehat F} \) is either a global
maximizer or a strict saddle point. Moreover, $\widehat F$ is locally geodesically strongly concave around each of its global maximizers. 
\end{theorem}
\begin{proof}
Obviously, it follows from \eqref{eq:final_condition} that \eqref{eq:delta_condition} holds. This, together with \Cref{prop:F_classify}, yields that \eqref{eq1:prop F class} and \eqref{eq2:prop F class} hold. Moreover, using \Cref{lem:dist_twosets} and \eqref{eq:final_condition}, we have \(\mathrm{dist}(\bm{Q}, \mathcal{Q}_G)\le r := \min\left\{\gamma_1,\gamma_2,1/2,\tau, {\sigma_0}/{3}\right\}\) for any $\bm Q \in \mathcal{Q}_{\widehat F}$. Thus every \( \bm Q\in\mathcal Q_{\widehat F}\) is \(r\)-close to \(\mathcal Q_{\rm opt}\) or to \(\mathcal Q_G\setminus\mathcal Q_{\rm opt}\).
When $\mathcal Q_G\setminus\mathcal Q_{\mathrm{opt}}\neq\emptyset$, the
two alternatives are mutually exclusive. Indeed, if both distance bounds held, then 
\begin{align*}
    \sigma_0 \overset{\eqref{eq:sigma0}}{\le} \operatorname{dist}(\bm Q,\mathcal Q_{\mathrm{opt}}) + \operatorname{dist}\bigl(\bm Q,\mathcal Q_G\setminus\mathcal Q_{\mathrm{opt}}\bigr) \le 2r \le \frac{2\sigma_0}{3},
\end{align*}
which is impossible. 

\textbf{Case 1.} Suppose that \( \bm{Q} \in \mathcal{Q}_{\widehat F} \) satisfies $\mathrm{dist}(\bm Q,\mathcal{Q}_{G}\setminus\mathcal{Q}_{\rm opt})\le r \le \gamma_2.$ This, together with \eqref{eq2:prop F class}, implies that $\bm Q$ is a strict saddle point, and thus not a global maximizer.

\textbf{Case 2.} Suppose that \( \bm{Q} \in \mathcal{Q}_{\widehat F} \) satisfies \({\mathrm{dist}\!(\bm Q,\mathcal{Q}_{\rm opt})\!\le\! r}\).
Since \(\mathcal Q_{\mathrm{opt}}\) consists of finitely many isolated points as defined in \eqref{eq:opti}, whose pairwise distances are at least \(2\) due to \(\bm s,\bm t\in\{1,-1\}^K\), and \(r\le 1/2<1\), there exists a unique \(\bm Q_\star\in\mathcal Q_{\mathrm{opt}}\) such that $\|\bm Q_\star-\bm Q\|_F = \min\left\{\|\bm U-\bm Q\|_F:\bm U\in\mathcal Q_{\mathrm{opt}}\right\}.$ By \(r\le \tau\), we have
\(\bm Q\in\mathcal B_{\rm geo}(\bm Q_\star,\eta)\).
Since \eqref{eq1:prop F class} holds throughout
\(\mathcal B_{\rm geo}(\bm Q_\star,\eta)\),
\Cref{lem:local_unique_critical} shows that
\(\bm Q\) is the unique critical point of \(\widehat F\) in this neighborhood.

It remains to show that \(\bm Q\) is a global maximizer. Suppose, for contradiction, that it is not. Since \(\mathrm{St}(d,K)\) is compact and \(\widehat F\) is continuous, there exists a global maximizer \(\widehat{\bm Q}\). Then \(\widehat{\bm Q}\in\mathcal Q_{\widehat F}\) and
$ \widehat F(\widehat{\bm Q}) > \widehat F(\bm Q).$
The point \(\widehat{\bm Q}\) cannot fall under \textbf{Case 1}, since every point in \textbf{Case 1} is a strict saddle point. Therefore, $    \operatorname{dist}(\widehat{\bm Q},\mathcal Q_{\mathrm{opt}})\le r.$
Using the same argument as above, there exists a unique $\widehat{\bm Q}_\star\in\mathcal{Q}_{\rm opt}$ such that $\|\widehat{\bm Q}_\star-\widehat{\bm Q}\|_F = \min\left\{\|\bm U-\widehat{\bm Q}\|_F:\bm U\in\mathcal{Q}_{\rm opt}\right\}$. By \Cref{prop:opt}, there exists \( \bm{s} \in \{1,-1\}^K \) such that $\bm Q_\star = \widehat{\bm Q}_\star \mathrm{diag}(\bm{s}).$ Moreover, by the definition of \( \widehat F\) in \eqref{eq:QPOC plugin} and \eqref{eq:Ak plugin}, we have
\begin{align}\label{eq:FQdiag}
    \widehat F(\bm Z) = \widehat F(\bm Z\mathrm{diag}(\bm{s})),\quad \forall \bm Z \in \mathrm{St}(d,K).
\end{align}
Thus, we obtain $\widehat F(\widehat{\bm Q}) = \widehat F(\widehat{\bm Q}\mathrm{diag}(\bm{s})).$ Therefore, \( \widehat{\bm Q}\mathrm{diag}(\bm{s}) \) is also a global maximizer by \eqref{eq:FQdiag}. Moreover, since \( \bm{s} \in \{1,-1\}^K \), we have $\| \bm Q_\star - \widehat{\bm Q} \mathrm{diag}(\bm{s}) \|_F=\|\widehat{\bm Q}_\star-\widehat{\bm Q}\|_F\le r.$ This implies {$\widehat{\bm Q}\mathrm{diag}(\bm{s}) \in\mathcal B_{\rm geo}(\bm Q_\star,\eta)$ because $r\le\tau$}. Moreover, \( \widehat{\bm Q}\mathrm{diag}(\bm{s})\neq \bm Q \); otherwise, \eqref{eq:FQdiag} would imply $\widehat F(\widehat{\bm Q})= \widehat F(\bm Q),$ contradicting \(\widehat F(\widehat{\bm Q}) > \widehat F(\bm Q)\). Thus, {$\mathcal B_{\rm geo}(\bm Q_\star,\eta)$} contains two distinct critical points, namely $\bm Q$ and $\widehat{\bm Q}\mathrm{diag}(\bm{s})$, contradicting the uniqueness established above. Hence, $\bm Q$ must be a global maximizer.

Finally, let $\bm Q$ be any global maximizer of $\widehat F$. Since a global maximizer is a critical point and cannot be a strict saddle point, it falls under \textbf{Case 2}. Hence, $\bm Q\in\mathcal B_{\rm geo}(\bm Q_\star,\eta)$ for some $\bm Q_\star\in\mathcal Q_{\mathrm{opt}}$. Choose a sufficiently small geodesically convex ball $\mathcal B_{\rm geo}(\bm Q,\varepsilon)$ contained in $\mathcal B_{\rm geo}(\bm Q_\star,\eta)$. It follows from \eqref{eq1:prop F class} that $\mathrm{Hess}\,\widehat F(\bm Z)[\bm D,\bm D] \le -\delta_1\|\bm D\|_F^2/3$  throughout this smaller ball. Therefore, $\widehat F$ is locally geodesically strongly concave around $\bm Q$. This completes the proof.
\end{proof}

It is worth noting that \cite[Lemma~SM3.14, supplementary material]{gilman2025semidefinite} and \cite[Lemma 3]{wang2023estimation} have established concentration bounds for the deviation of the sample covariance matrix of Gaussian random vectors from its expectation as follows. Recall that $\bm y_{l,i} \sim \mathcal{N}\left(\bm 0, \bm Q^\star { \bLambda} \bm Q^{\star T} + v_l\bm I_d\right)$ as defined in \eqref{eq:HPCA} and define
\begin{align}\label{eq:Cl}
    \bm C_l := \frac{1}{n}\sum_{i=1}^{n_l} \bm y_{l,i}\bm y_{l,i}^T,\qquad
    \xi_{l} := \frac{\sum_{k=1}^{K}\lambda_{k}+v_{l}d}{\lambda_1+v_l}. 
\end{align}

\begin{lemma}
 \label{lem:delta_lemma}   
For each fixed $l\in[L]$ and each $t>0$, it holds with probability at least $1-e^{-t}$ that
\begin{equation*}
\|\bm C_{l}-\mathbb{E}\left[\bm{C}_{l}\right]\| \leq \frac{c_1 n_l}{n}(\lambda_1+v_l)\max\left\{\sqrt{\frac{\xi_l \log d+t}{n_l}}, \frac{\xi_{l}\log d+t}{n_l}\log n\right\},
\end{equation*}
where  $c_{1}>0$ is a universal constant.
\end{lemma}

By setting $t=\log n$ in the above lemma and applying the union bound, we obtain the following corollary. 
\begin{corollary}\label{coro:concen}
    It holds with probability at least $1-L/n$ that for all $l \in [L]$,
    \begin{align}\label{eq:cov}
\|\bm C_{l}-\mathbb{E}\left[\bm{C}_{l}\right]\| \le \frac{c_1 n_l(\lambda_1+v_l)}{n}\max\left\{\sqrt{\frac{\xi_l \log d + \log n}{n_l}}, \frac{\xi_{l}\log d + \log n}{n_l}\log n\right\},
\end{align}
where $c_{1}>0$ is a universal constant.
\end{corollary}
Before proceeding, we note that for every $l\in[L]$,
\begin{equation}\label{eq:xi}
    \xi_l = \frac{\sum_{k=1}^K \lambda_k + v_l d}{\lambda_1+v_l}
    \le
    \frac{d\lambda_1+v_l d}{\lambda_1+v_l}
    = d,
\end{equation} 
where the inequality follows from $\sum_{k=1}^K\lambda_k\le K\lambda_1\le d\lambda_1$. Using \Cref{thm:Delta} and \Cref{coro:concen}, we are ready to show \Cref{thm:theorem5}, which establishes the optimization landscape and geometry of the HePPCA problem \eqref{eq:QPOC plugin} when the number of training samples satisfies a suitable lower bound. 

\begin{proof}[Proof of \Cref{thm:theorem5}]
Suppose that \eqref{eq:cov} holds for each $l \in [L]$, which happens with probability at least $1-L/n$ according to \Cref{coro:concen}. According to \eqref{eq:Ak plugin}, \eqref{eq:Delta}, and \eqref{eq:Cl}, we have
\[\|\bm \Delta_k\|= \left\|\sum_{l\in[L]} \frac{\widehat w_{l,k}}{\widehat v_l}(\bm C_l-\E[\bm C_l])\right\|.
\]
This implies 
\begin{align*} 
    \sum_{k=1}^K \|\bm \Delta_k\|
    & \le \sum_{k=1}^K\sum_{l=1}^L \frac{\widehat w_{l,k}}{\widehat v_l}\left\| \bm C_l -\mathbb{E}[\bm C_l] \right\| \\
    & \overset{\eqref{eq:cov}}{\le} \frac{c_1}{n} \sum_{k=1}^K\sum_{l=1}^L \frac{\widehat w_{l,k}n_l}{\widehat v_l}(\lambda_1+v_l)
    \max\left\{\sqrt{\frac{\xi_l \log d + \log n}{n_l}}, \frac{\xi_{l}\log d + \log n}{n_l}\log n\right\} \\
    & \le \frac{c_1\widehat{\beta}}{n} \sum_{l=1}^L n_l
    \max\left\{\sqrt{\frac{d \log d + \log n}{n_l}}, \frac{d \log d + \log n}{n_l}\log n\right\} \\
    & \le \frac{c_1\widehat{\beta}}{n} \left(\sqrt{d\log d+\log n} \sqrt{Ln} + L(d\log d+\log n)\log n\right),
\end{align*}
where the third inequality uses $\widehat\beta := \max_{l\in[L]}\frac{\lambda_1+v_l}{\widehat v_l}\sum_{k=1}^K\widehat w_{l,k}$ and the last inequality follows from $\sum_{l=1}^L \sqrt{n_l} \le \sqrt{nL}$ and $\max\{a,b\}\le a+b$. Using the above inequality and \eqref{eq:n} yields
\begin{align*}
    \sum_{k=1}^K \|\bm \Delta_k\| \le c_1\widehat\beta
    \sqrt{\frac{L(d\log d+\log n)}{n}}
    +\frac{c_1\widehat\beta L(d\log d+\log n)\log n}{n} \overset{\eqref{eq:n}}{\le} \frac{\delta_3}{2} + \frac{\delta_3}{2}  \le \delta_3.   
\end{align*}
This, together with \Cref{thm:Delta}, implies the desired result.
\end{proof}

\begin{figure*}[t]
\begin{center}
	\begin{subfigure}{0.32\textwidth}
    	\centering\includegraphics[width = 1\linewidth]{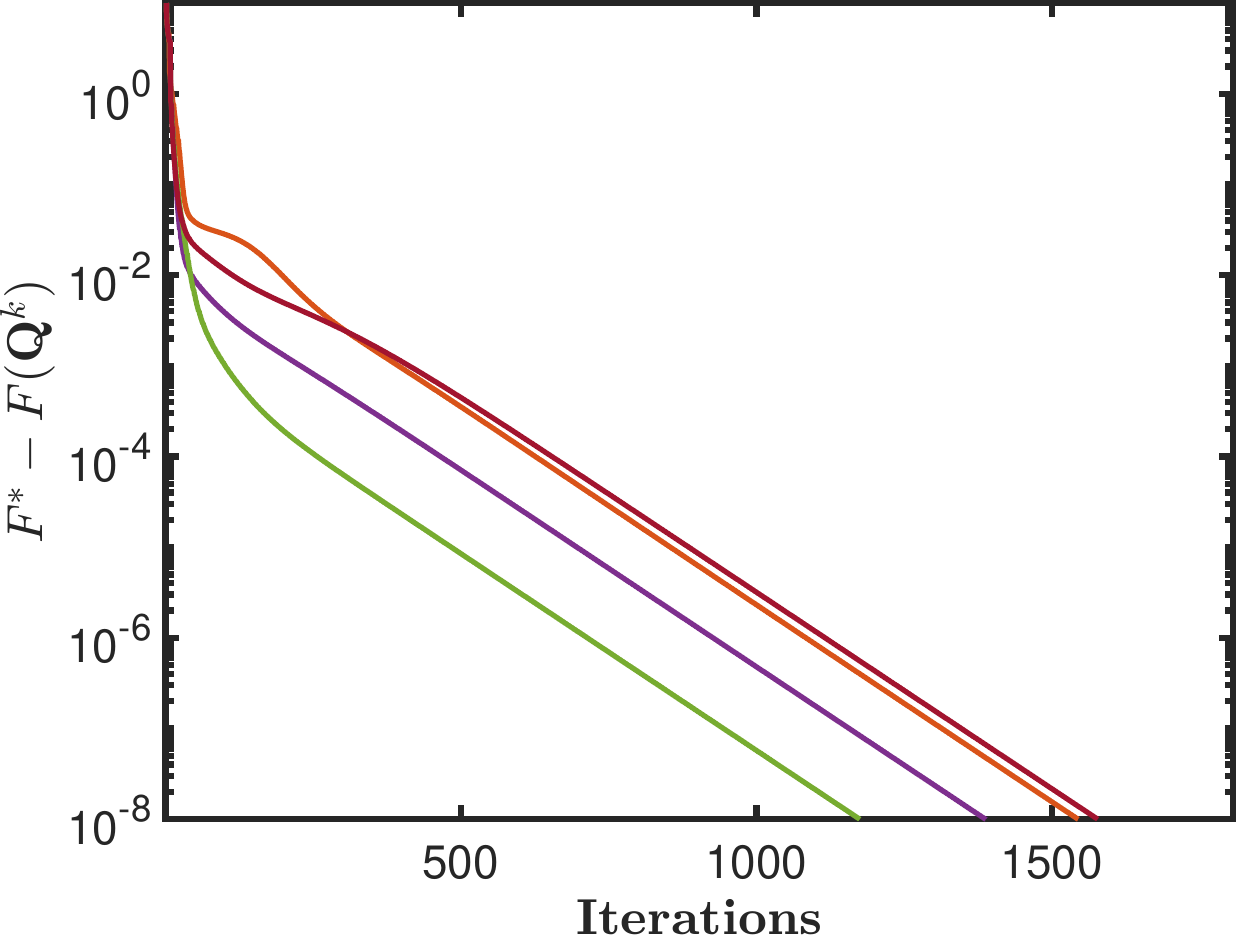}  
    \caption{$d=100, K=20$} 
    \end{subfigure} \hfill
    \begin{subfigure}{0.32\textwidth}
    	\centering\includegraphics[width = 1\linewidth]{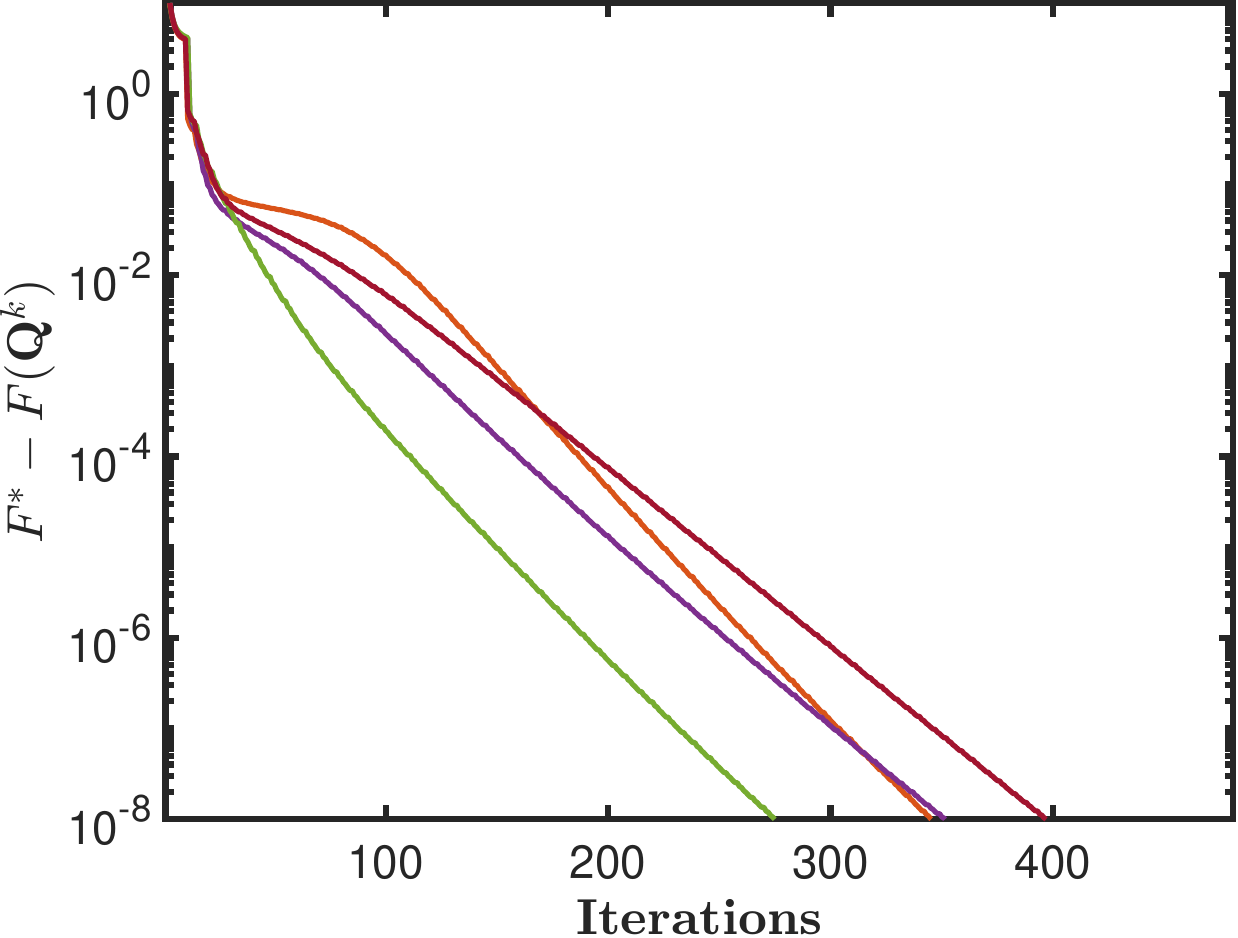}  
    \caption{$d=100, K=40$} 
    \end{subfigure}\hfill
    \begin{subfigure}{0.32\textwidth}
    	\centering\includegraphics[width = 1\linewidth]{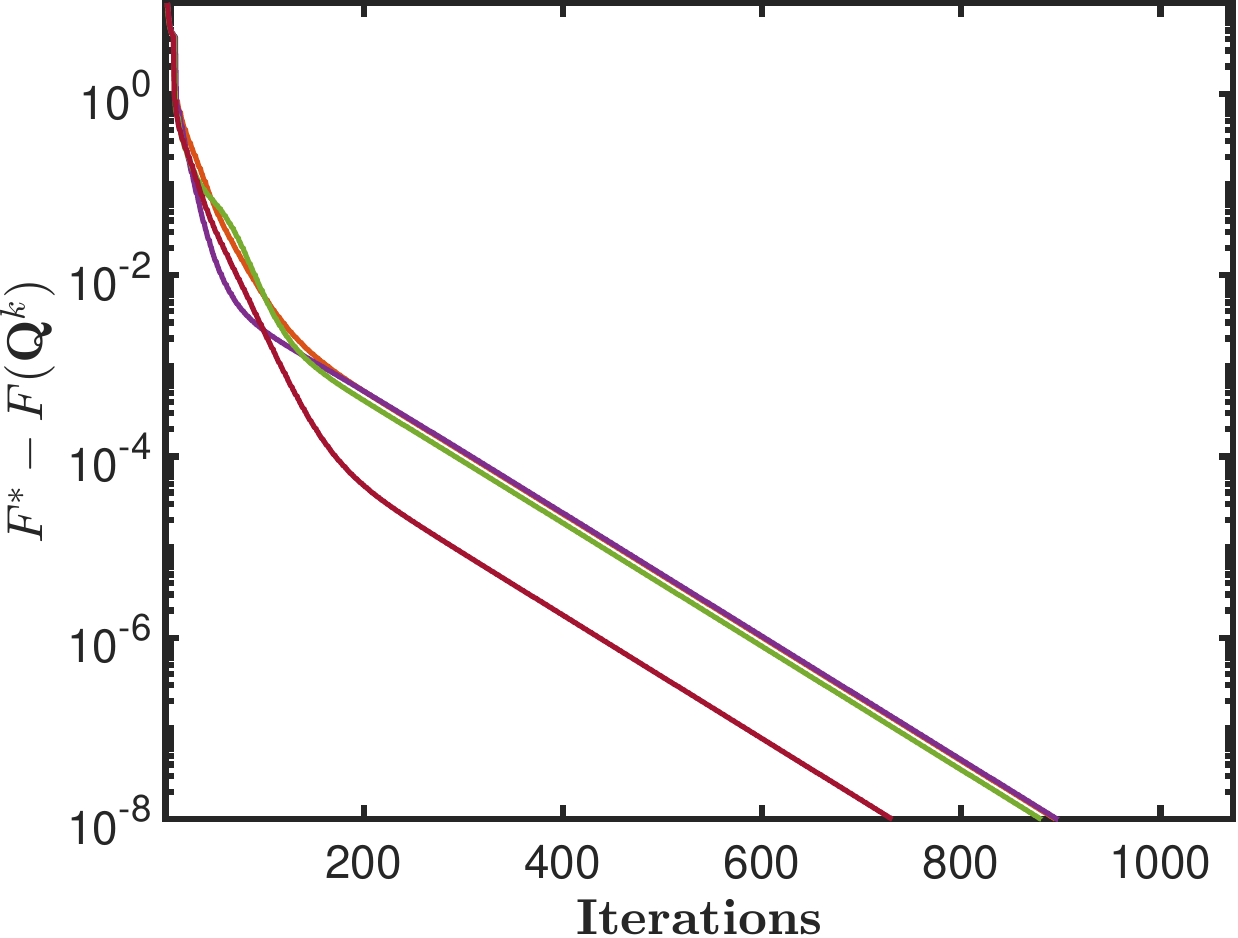}  
    \caption{$d=100, K=60$} 
    \end{subfigure}
    \caption{\textbf{Convergence behavior of Riemannian gradient ascent (RGA) for solving QPOC.} The $x$-axis is number of iterations $k$, and $y$-axis is the function value gap $F^*-F^k$. Here, \(F^k:=F(\bm Q^k)\) denotes the function value at the \(k\)-th iterate \(\bm Q^k\) generated by RGA, and \(F^*\) is the optimal value of Problem~\eqref{eq:QPOC}. 
    } 
    \label{fig:1}
\end{center}
\end{figure*}

\section{Experimental Results}\label{sec:exp} 
 
In this section, we conduct numerical experiments on the QPOC and HePPCA problems to illustrate the algorithmic implications of our landscape and geometry results.

\subsection{Experiments on QPOC}

In this subsection, we conduct numerical experiments on synthetic data to support \Cref{thm:land}. To this end, we apply Riemannian gradient ascent (RGA) to solve {Problem~\eqref{eq:QPOC}} and examine whether it converges to a global maximizer from different random initializations. For each random initialization, we generate a Gaussian random matrix $\bm Z \in \R^{d\times K}$ and compute its compact SVD $\bm Z = \bm U \bm \Sigma \bm V^T$, where $\bm U \in \mathrm{St}(d,K)$ and $\bm V \in \mathcal{O}^{K}$. We then set $\bm Q^0 = \bm U \bm V^T$.

In our experiments, we set $d=100$ and $K \in \{20,40,60\}$, and generate the matrices $\bm A \in \S^{d}$ and $\bm B \in \S^K$ as follows. Specifically, we generate a vector $\bm a \in \R^d$ by sampling its entries independently from the uniform distribution on $[5,10]$ and then sorting them in descending order. Similarly, we generate a vector $\bm b \in \R^K$ by sampling the first $K/2$ entries from the uniform distribution on $[1,1.1]$ and the remaining entries from the uniform distribution on $[-1,-0.9]$, after which the entries are sorted in descending order. We then set $\bm A = \mathrm{diag}(\bm a)$ and $\bm B = \mathrm{diag}(\bm b)$.
 
For each problem instance, we keep $\bm A$ and $\bm B$
fixed and run RGA from four random initial points. The algorithm uses a polar retraction computed by SVD and a backtracking line search with initial trial step size
$1$ and contraction factor $0.5$. Then, we visualize the convergence behavior in terms of the optimality gap \(F^*-F^k\) for solving {Problem~\eqref{eq:QPOC}}; see \Cref{fig:1}. {Here, \(F^k:=F(\bm Q^k)\) denotes the function value at the \(k\)-th iterate \(\bm Q^k\) generated by RGA, and \(F^*\) is the optimal value of {Problem~\eqref{eq:QPOC}}, computed according to \eqref{eq:FQ}.} It can be observed that, for all four initializations tested, RGA consistently converges to a global maximizer at a linear rate. This behavior is empirically consistent with the theoretical results in \Cref{thm:land} regarding the benign optimization landscape of {Problem~\eqref{eq:QPOC}}.

\begin{figure*}[t]
\begin{center}
	\begin{subfigure}{0.32\textwidth}
    	\centering\includegraphics[width = 1\linewidth]{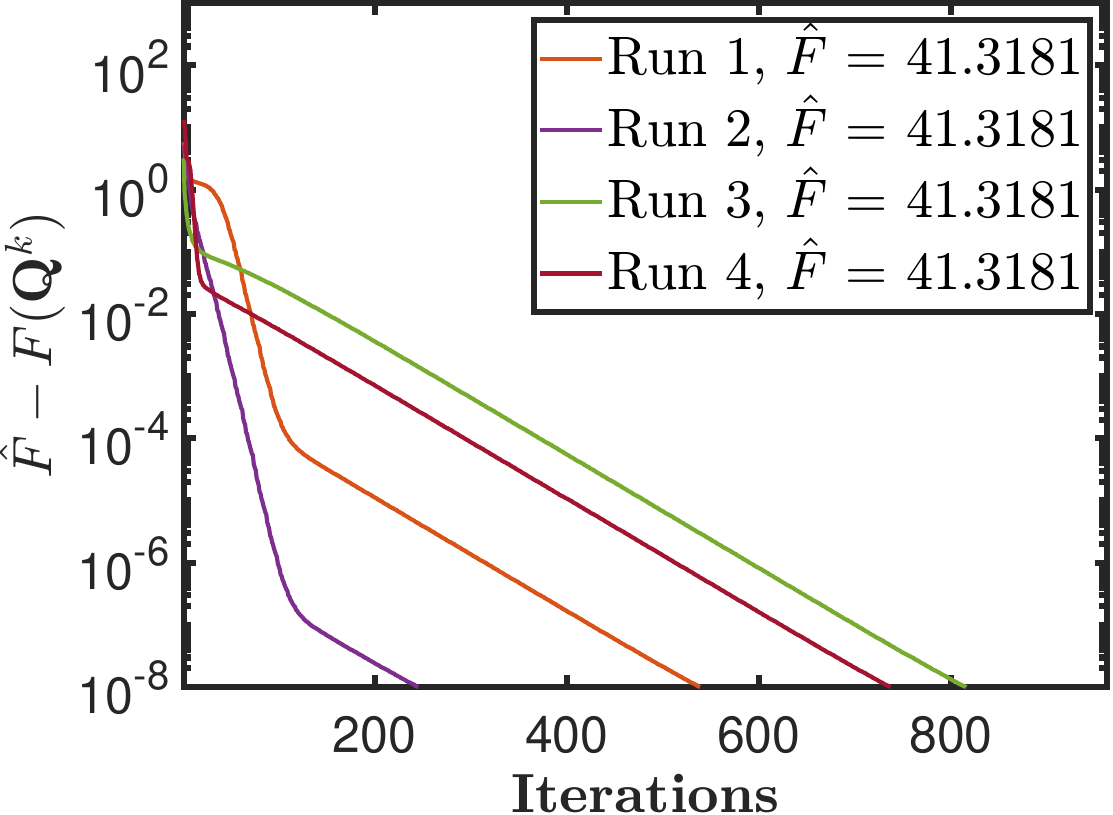}  
    \caption{First set of samples} 
    \end{subfigure} 
    \begin{subfigure}{0.32\textwidth}
    	\centering\includegraphics[width = 1\linewidth]{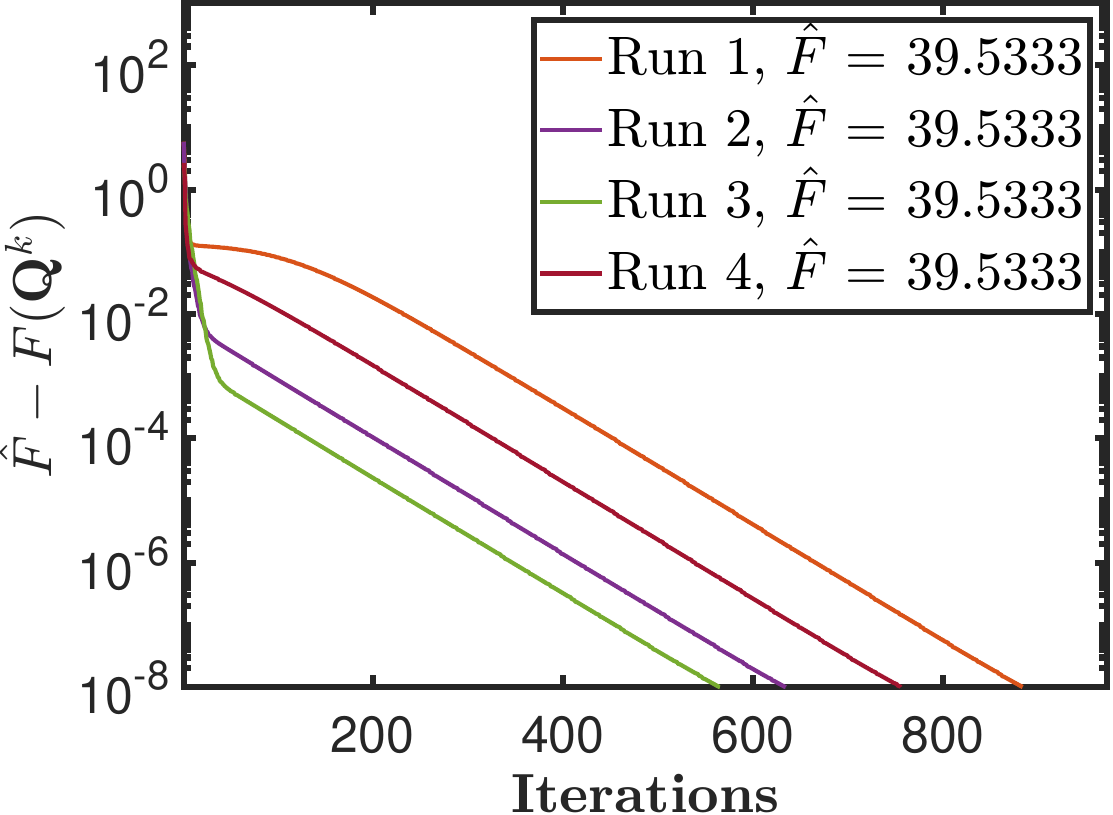}  
    \caption{Second set of samples} 
    \end{subfigure} 
    \begin{subfigure}{0.32\textwidth}
    	\centering\includegraphics[width = 1\linewidth]{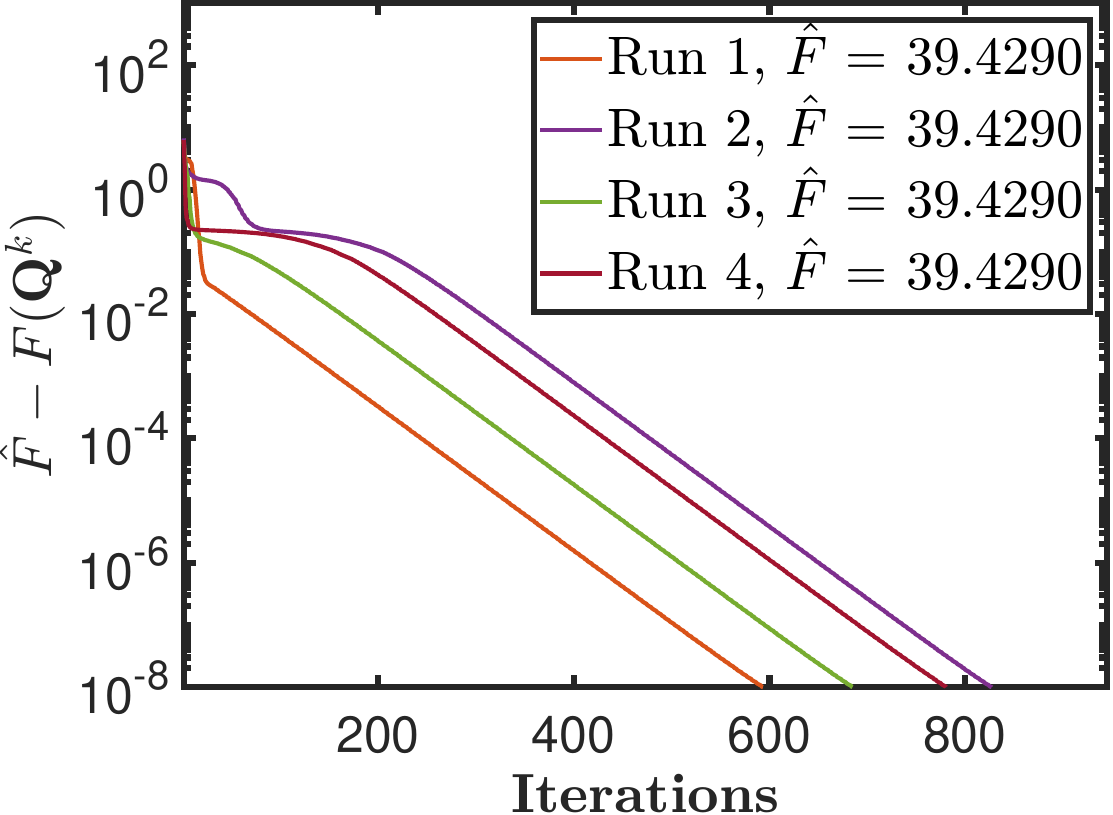}  
    \caption{Third set of samples} 
    \end{subfigure}
    \caption{\textbf{Convergence behavior of RGA for HePPCA with exact parameters.} The $x$-axis is number of iterations $k$, and $y$-axis is the function value gap $\hat F - F^k$. Here, $F^k=F(\bm Q^k)$ denotes the function value at the $k$-th iterate $\bm Q^k$, and $\hat F$ denotes the function value at the final iterate. 
    }
    \label{fig:2}
\end{center}
\end{figure*}

\subsection{Experiments on HePPCA}

In this subsection, we investigate the convergence behavior of RGA for HePPCA with exact parameters and fixed plug-in estimates. We generate data samples $\{\bm y_{l,i}\}$ according to the HePPCA model in \eqref{eq:HPCA}, with total dimension $d=20$, subspace dimension $K=3$, and $L=2$ data groups. Each group contains $n_l=50$ samples, giving a total sample size of $n=100$. The true signal strengths
and noise variances are $\bm\lambda=(3,2,0.5)$ and $\bm v=(0.2,0.1)$. The ground-truth matrix $\bm Q^\star\in\mathrm{St}(d,K)$ is obtained as the orthogonal factor in the reduced QR decomposition of a matrix with i.i.d. standard normal entries. 

For each dataset, we run RGA from four random initial points, each generated using the same Gaussian--QR construction. We use the Riemannian metric induced by the Euclidean inner product, a QR retraction, and Armijo backtracking with initial
trial step size $1$, contraction factor $0.5$, and sufficient-decrease parameter $0.1$. The algorithm terminates when the Frobenius norm of the Riemannian gradient falls below $10^{-6}$ or when $10^4$ iterations have been performed.

\begin{figure*}[t]
\begin{center}
	\begin{subfigure}{0.32\textwidth}
    	\centering\includegraphics[width = 1\linewidth]{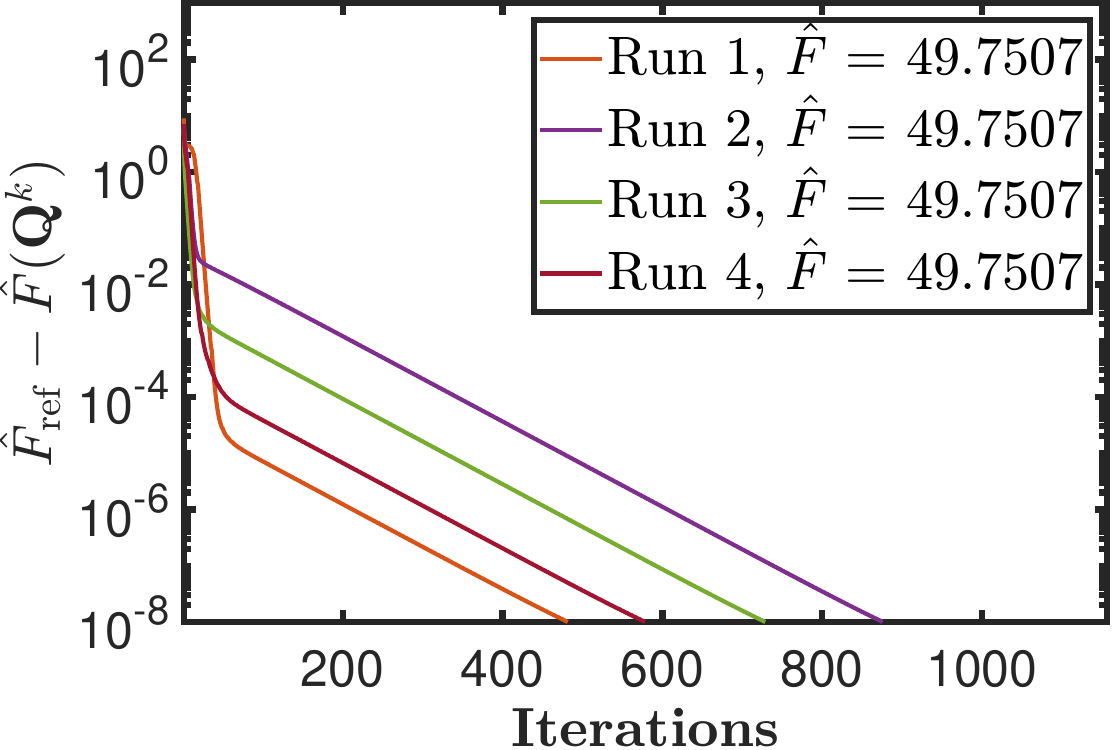}  
    \caption{First set of samples} 
    \end{subfigure} 
    \begin{subfigure}{0.32\textwidth}
    	\centering\includegraphics[width = 1\linewidth]{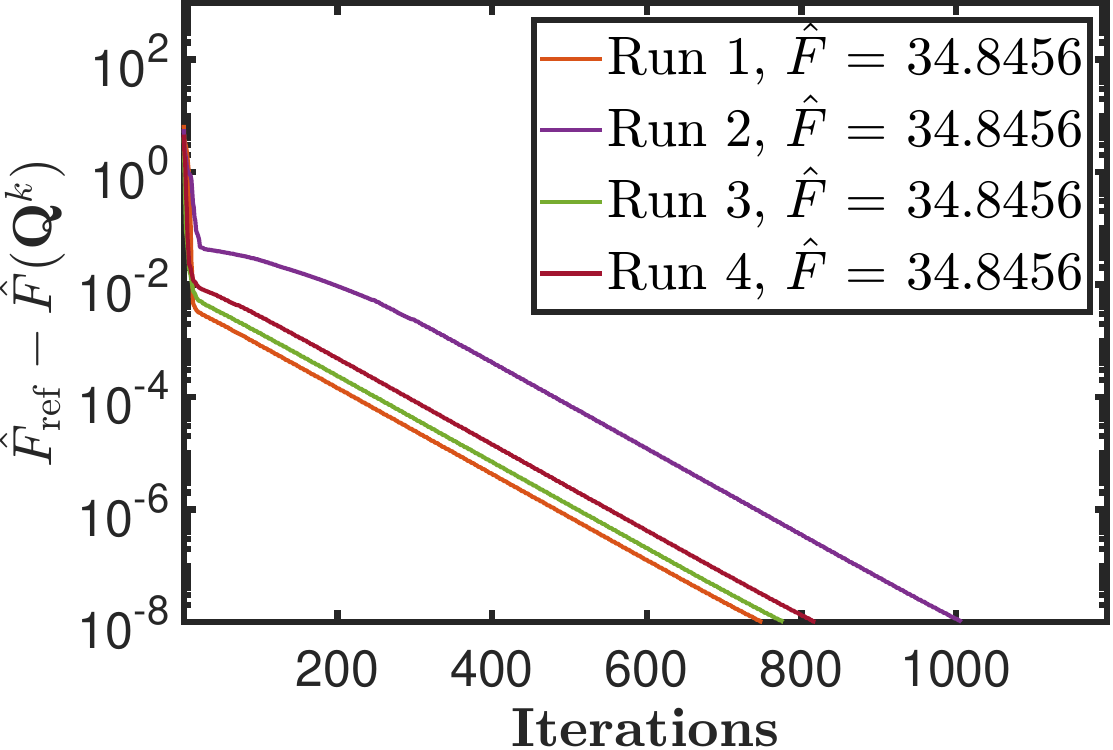}  
    \caption{Second set of samples} 
    \end{subfigure} 
    \begin{subfigure}{0.32\textwidth}
    	\centering\includegraphics[width = 1\linewidth]{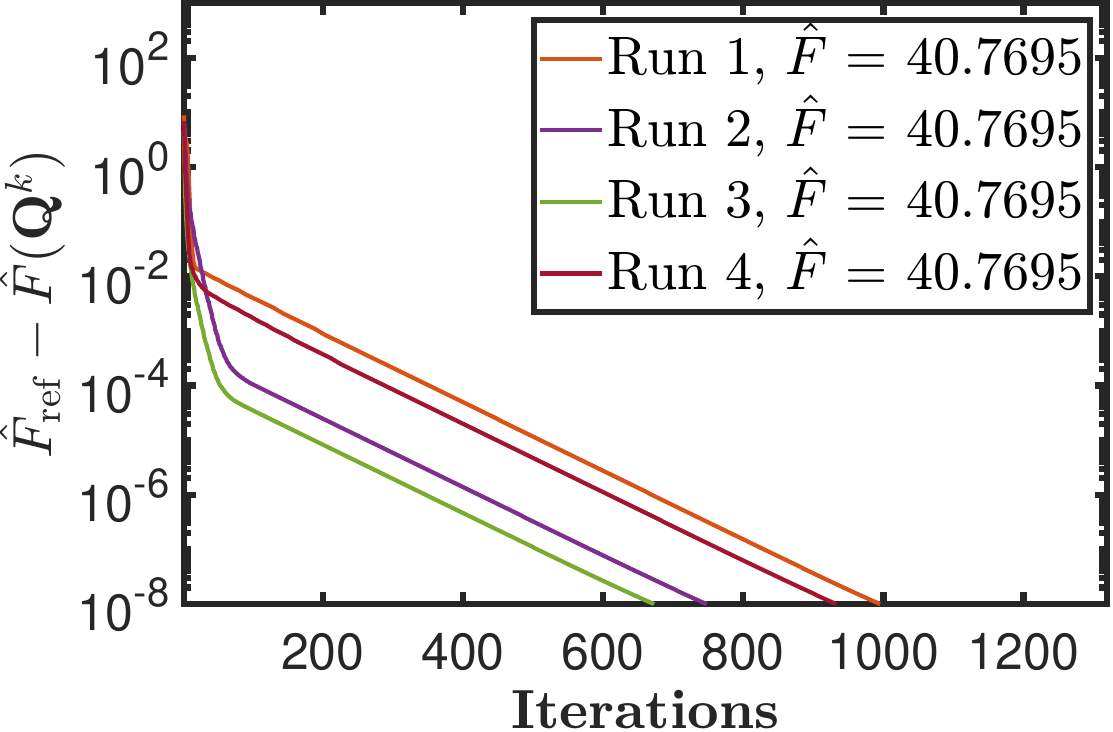}  
    \caption{Third set of samples} 
    \end{subfigure}
    \caption{\textbf{Convergence behavior of RGA for HePPCA with plug-in estimates.} The $x$-axis is number of iterations $k$, and $y$-axis is the function value gap $\hat F - F^k$. Here, $F^k=F(\bm Q^k)$ denotes the function value at the $k$-th iterate $\bm Q^k$, and $\hat F$ denotes the function value at the final iterate.
    }
    \label{fig:4}
\end{center}
\end{figure*}

\vspace{-0.1in}
\paragraph{HePPCA with exact parameters.} We first consider Problem~\eqref{eq:HePPCA} with the true signal strengths and noise variances used to construct the coefficient matrices $\{\bm A_k\}_{k=1}^K$ according to \eqref{eq:Ak}. We generate three independent problem instances by repeating
the ground-truth and sample generation procedures. For each instance, the same coefficient matrices are used for all four initializations.  We plot the objective gaps $\hat{F} - F(\bm Q^k)$ in \Cref{fig:2}, where $\hat{F}$ denotes the function value at the final iterate. For each problem instance, the four runs attain numerically consistent objective values, and the decay of the objective gaps is consistent with local linear convergence. These observations provide empirical evidence consistent with the benign global landscape and local geometry in \Cref{thm:theorem5}.

\vspace{-0.1in}
\paragraph{HePPCA with plug-in estimates.} We next consider Problem~\eqref{eq:QPOC plugin} with plug-in estimates. We fix one realization of $\bm Q^\star$ and generate three independent datasets using the true parameters $\bm\lambda$ and $\bm v$. For each dataset, we estimate plug-in estimates $\widehat{\bm \lambda}$ and $\widehat{\bm v}$ as follows. First, we perform PCA on all samples by forming the sample covariance matrix $ \widehat{\bm\Sigma} := \sum_{l=1}^L \bm Y_l\bm Y_l^T/n.$ Let $\widehat\mu_1\ge\cdots\ge\widehat\mu_d$ denote its eigenvalues, and let $\widehat{\bm Q}_{\mathrm{PCA}}\in\mathrm{St}(d,K)$ contain its leading $K$ eigenvectors. Next, we estimate each group's noise variance using the residual energy outside the estimated signal subspace: 
\begin{align*}
        \widehat v_l
    := \frac{\|(\bm I_d-\widehat{\bm Q}_{\mathrm{PCA}}
    \widehat{\bm Q}_{\mathrm{PCA}}^T)\bm Y_l\|_F^2}
    {n_l(d-K)},
    \quad \forall l\in[L].
\end{align*}
Finally, we estimate the signal strengths by subtracting the estimated average noise variance from the leading eigenvalues:
\[
    \widehat\lambda_k
    := \widehat\mu_k-\sum_{l=1}^L\frac{n_l}{n}\widehat v_l,\quad \forall k \in [K].
\]
For each dataset, we construct
the empirical coefficient matrices by replacing $\lambda_k$ and $v_l$ in \eqref{eq:Ak} with $\widehat\lambda_k$ and $\widehat v_l$, respectively. 

For each of the three datasets, \Cref{fig:4} shows the gaps $\hat{F} - \hat F(\bm Q^k)$. Although both $\widehat{\bm\lambda}$ and $\widehat{\bm v}$ are estimated from the same dataset used to construct the plug-in objective, the objective gaps still decay at a rate consistent with local linear convergence. These results illustrate favorable convergence behavior even when the fixed plug-in estimates are substantially inaccurate, in agreement with the benign global landscape and local geometry in \Cref{thm:theorem5}.

\section{Conclusion}\label{sec:con} 

In this work, we characterized the global optimization landscape and local geometry of HePPCA with fixed plug-in estimates of the signal strengths and noise variances. Our analysis builds on a constructive study of QPOC, which provides a necessary and sufficient condition for global optimality within the critical point set and explicit positive-curvature directions at all nonglobal critical points. Applying these results to population HePPCA, we established its benign landscape and quantitative curvature bounds, including local geodesic strong concavity near every global maximizer. We then combined a deterministic perturbation analysis with sample covariance
concentration bounds to show that empirical HePPCA inherits
these properties with high probability under a sufficient sample-size condition. One future direction is to extend the constructive analysis to nonhomogeneous QPOC problems with an additional linear term, as well as other manifold optimization problems arising in machine learning, such as maximal coding rate reduction \cite{chan2022redunet} and neural collapse \cite{yaras2022neural}. Another direction is to establish a benign landscape and favorable local geometry for HePPCA when the plug-in estimates and empirical objective are constructed from the same data, accounting for their statistical dependence. 

\bibliographystyle{abbrvnat}
\bibliography{reference}

\end{document}